%% file: main.tex
\documentclass[oribibl,a4paper,envcountsame]{llncs}

\newif\ifcomments
\commentsfalse  % change this to false to hide our rainbow comments :-)

\newif\iffinal
\finalfalse      % Crypto final version omits appendices, hardcodes date, etc.

\title{
    Improved torsion-point attacks on SIDH variants
}

\author{
    Victoria de Quehen\inst{1}
    \and
    P\'{e}ter Kutas\inst{2}
    \and
    Chris Leonardi\inst{1}
    \and
    Chloe Martindale\inst{3}
    \and \\
    Lorenz Panny\inst{4}
    \and
    Christophe Petit\inst{2,5}
    \and
    Katherine E. Stange\inst{6}
}

\institute{
    ISARA Corporation, Waterloo, Canada
    \\ 
        \email{vicdequehen@gmail.com}
        \\
        \email{chris.leonardi@isara.com}
    \and
    {School of Computer Science, University of Birmingham, UK}
        \\
        \email{P.Kutas@bham.ac.uk}
    \and
    Department of Computer Science,
    University of Bristol, UK
        \\
        \email{chloe.martindale@bristol.ac.uk}
    \and
    Institute of Information Science, Academia Sinica, Taipei, Taiwan
        \\
        \email{lorenz@yx7.cc}
    \and
        Laboratoire d'Informatique, \\
    Universit\'e libre de Bruxelles, Belgium
        \\
        \email{christophe.f.petit@gmail.com}
    \and
    {Department of Mathematics,
    University of Colorado Boulder, Colorado, USA}
        \\
        \email{kstange@math.colorado.edu}
}

\newcommand\None{A}
\newcommand\Ntwo{B}

\newcommand\Ostar{O^{\hspace{-.1ex}{\ast}\hspace{-.2ex}}}

\usepackage[dvipsnames]{xcolor}
\definecolor{linkcolor}{rgb}{0.65,0,0}
\definecolor{citecolor}{rgb}{0,0.65,0}
\definecolor{urlcolor}{rgb}{0,0,0.65}
\usepackage[
        pdfauthor={Péter Kutas, Chloe Martindale, Lorenz Panny, Christophe Petit, Victoria de Quehen, Kate E. Stange, Chris Leonardi},
        pdftitle={Improved torsion-point attacks on SIDH variants},
    colorlinks=true, linkcolor=linkcolor, urlcolor=urlcolor, citecolor=citecolor
]{hyperref}
\usepackage[hyphenbreaks]{breakurl} %arXiv messes up URLs without this

\definecolor{myyellow}{rgb}{1.0, 0.75, 0.0}
\definecolor{mygreen}{rgb}{0.35, 0.71, 0.0}

\usepackage[utf8]{inputenc}
\usepackage[T1]{fontenc}
\usepackage{amsmath,mathtools}
\usepackage{amssymb}
\usepackage{mathrsfs}
\usepackage{hyperref}
\usepackage[style=english]{csquotes}
\usepackage{graphicx}
\usepackage{tikz-cd}
\usepackage{enumerate}
\usepackage{lipsum}
\usepackage[british]{babel}
\addto{\captionsbritish}{\renewcommand\abstractname{Abstract.}}
\usepackage{tipa}
\usepackage{mathtools}
\usepackage{amssymb}
\usepackage{xcolor}
\usepackage{nth}
\usepackage{tikz}
\usepackage{tikz-cd}
\usepackage{comment}
\usepackage[open,openlevel=1]{bookmark}
\usepackage{multirow}
\usepackage{appendix}
\usepackage{listings}
\usepackage{setspace}
\usepackage{pifont}
\usepackage{bm}
\usepackage{adjustbox}

\usepackage{tkz-euclide}

\usepackage{booktabs}
\usepackage{multirow}

\usepackage{enumitem}
\setitemize  {topsep=1ex,itemsep=.5ex}
\setenumerate{topsep=1ex,itemsep=.5ex}

\usepackage[ruled,vlined,linesnumbered]{algorithm2e}
\usepackage{algpseudocode}
\SetEndCharOfAlgoLine{}
\DecMargin{2.2mm}
\SetKwInput{Input}{Input}\SetKwInput{Output}{Output}
\SetKwFor{While}{While}{do}{}
\SetKwFor{ForEach}{For each}{do}{}
\SetKwIF{If}{ElseIf}{Else}{If}{then}{else if}{else}{endif}
\SetKw{Return}{Return}
\SetKw{KwRet}{return}
\SetKw{Recurse}{Recurse}
\SetKwFor{For}{For}{do}{}
\SetKw{Break}{break}

\newcommand\mybullet{\raisebox{.22ex}{\smaller\smaller\,\textbullet\hspace{.04em}}}

\newlength\InputNewLineIndent
\newcommand\InputNewLine{%
    \settowidth\InputNewLineIndent{\KwIn{}}%
    \\%
    \mbox{}%
    \hspace*{\InputNewLineIndent}%
    \relax%
}

\usepackage{geometry}

\newcommand\ul[1]{\underline{\smash{#1}}}

\usepackage{relsize}

\spnewtheorem{heuristic}{Heuristic}{\normalfont\bfseries}{\itshape}
\undef\problem\spnewtheorem{problem}{Problem}{\normalfont\bfseries}{\normalfont}

\usepackage{xspace}
\newcommand{\jinvariant}{\(j\)\babelhyphen{nobreak}invariant\xspace}
\newcommand\Z{\mathbb Z}
\newcommand\Q{\mathbb Q}
\newcommand\F{\mathbb F}
\newcommand\BB{\mathcal B}
\newcommand\OO{\mathcal O}
\DeclareMathOperator{\End}{End}
\DeclareMathOperator{\Hom}{Hom}
\DeclareMathOperator{\real}{Re}
\DeclareMathOperator{\ima}{Im}
\DeclareMathOperator{\nor}{N}

\DeclareMathOperator{\tr}{tr}
\DeclareMathOperator{\SLE}{SLE}

\renewcommand\phi\varphi    % unify

\renewcommand\i{\mathbf{i}}
\renewcommand\j{\mathbf{j}}
\renewcommand\ij{\mathbf{ij}}
\newcommand\ji{\mathbf{ji}}

\newcommand\cl{\mathrm{cl}}
\usepackage{trimclip}
\newcommand\Ell{%
    \mathcal{E}%
    \hspace{-.35ex}%
    \clipbox{.25mm 0mm 0mm 0mm}{$\ell$}%
    \hspace{-.18ex}%
    \raisebox{-.05mm}{\rotatebox{2}{\clipbox{.15mm 0mm 0mm 0mm}{$\ell$}}}%
    }

\renewcommand\hat\widehat %XXX hack
\newcommand\dual\widehat
\newcommand\bigO{O}
\newcommand\polylog{\operatorname{polylog}}
\newcommand{\bd}{[d]}

\newcommand\comp{\mathscr C}

\let\oldlabelitemii\labelitemii
\let\labelitemii\labelitemi
\let\labelitemi\oldlabelitemii

\ifcomments
\newcommand{\CV}[1]{\textcolor{violet}{{\sf (Vic:} {\sl{#1})}}}
\newcommand{\CL}[1]{\textcolor{CornflowerBlue}{{\sf (Chris:} {\sl{#1})}}}
\newcommand{\CP}[1]{\textcolor{blue}{{\sf (Christophe:} {\sl{#1})}}}
\newcommand{\PK}[1]{\textcolor{ForestGreen}{{\sf (Peter:} {\sl{#1})}}}
\newcommand{\KS}[1]{\textcolor{Bittersweet}{{\sf (Kate:} {\sl{#1})}}}
\newcommand{\CM}[1]{\textcolor{orange}{{\sf (Chloe:} {\sl{#1})}}}
\newcommand{\LP}[1]{\textcolor{BlueGreen}{{\sf (Lorenz:} {\sl{#1})}}}
\else
\newcommand{\CV}[1]{}
\newcommand{\CL}[1]{}
\newcommand{\CP}[1]{}
\newcommand{\PK}[1]{}
\newcommand{\KS}[1]{}
\newcommand{\CM}[1]{}
\newcommand{\LP}[1]{}
\fi

\newcommand{\erase}[1]{}

\newcommand\frst{%
        &\mbox{}%
        \hspace{.9ex}
    }
\newcommand\cont{%
        \hspace{.3ex}%
        \textcolor{gray}{...}%
        \\ &\mbox{}%
        \hspace{-1.5ex}%
        \textcolor{gray}{...}%
        \hspace{.4ex}
    }
\makeatletter
\def\namedlabel#1#2{\begingroup
   \def\@currentlabel{#2}%
   \label{#1}\endgroup
}
\makeatother

\begin{document}

\maketitle

\begingroup
\makeatletter
\def\@thefnmark{$\ast$}\relax
\@footnotetext{\relax
Author list in alphabetical order; see
\url{https://www.ams.org/profession/leaders/culture/CultureStatement04.pdf}.
Lorenz Panny was a PhD student at Technische Universiteit Eindhoven
while this research was conducted.
P\'eter Kutas and Christophe Petit's work was supported by  EPSRC grant EP/S01361X/1.  Katherine E. Stange was supported by NSF-CAREER CNS-1652238.
  This work was supported in part
  by the Commission of the
  European Communities through the Horizon 2020 program under project number
  643161 (ECRYPT-NET)
  and in part by NWO project 651.002.004 (CHIST-ERA USEIT).
\def\ymdtoday{\leavevmode\hbox{\the\year-\twodigits\month-\twodigits\day}}\def\twodigits#1{\ifnum#1<10 0\fi\the#1}%
\iffinal
Date of this document: 2021-06-25.
``\textcopyright IACR 2021. This article is the final version submitted by the author(s) to the IACR and to Springer-Verlag on June 25, 2021. The version published by Springer-Verlag is available at <DOI>.''
\else
Date of this document: \ymdtoday.
\fi
}
\endgroup

%\tableofcontents
%\newpage

\begin{abstract}

SIDH is a post-quantum key exchange algorithm based on the presumed difficulty of finding isogenies between supersingular elliptic curves.
However, SIDH and related cryptosystems also reveal additional information: the restriction of a secret isogeny to a subgroup of the curve (torsion-point information). Petit \cite{Petit2017} was the first to demonstrate that torsion-point information could noticeably lower the difficulty of finding secret isogenies. 
In particular, Petit showed that ``overstretched'' parameterizations of SIDH could be broken in polynomial time. However, this did not impact the security of any cryptosystems proposed in the literature.
The contribution of this paper is twofold: First, we strengthen the techniques of \cite{Petit2017} by exploiting additional information coming from a dual and a Frobenius isogeny. This extends the impact of torsion-point attacks considerably. In particular, our techniques yield a classical attack that completely breaks the $n$-party group key exchange of \cite{DBLP:journals/iacr/AzarderakhshJJS19}, first introduced as GSIDH in \cite{DBLP:conf/isita/FurukawaKT18}, for 6 parties or more, and a quantum attack for 3 parties or more that improves on the best known asymptotic complexity. We also provide a Magma implementation of our attack for 6 parties. We give the full range of parameters for which our attacks apply.
Second, we construct SIDH variants designed to be weak against our attacks; this includes backdoor choices of starting curve, as well as backdoor choices of base-field prime.
We stress that our results do not degrade the security of, or reveal any weakness in, the NIST submission SIKE \cite{azarderakhsh2019supersingular}.

\end{abstract}

\setcounter{footnote}{0}

\input{intro}

\subsubsection*{Acknowledgements.}

\noindent
Thanks to Daniel J.\ Bernstein for his insight into estimating sizes of solutions to Equation~\ref{the_eqnofdeath}, to John Voight for answering a question of ours concerning Subsection~\ref{sec:numweak}, and to Boris Fouotsa for identifying errors in Proposition~\ref{prop:nonpoly_unbalanced} and its proof.
We would also like to thank Filip Pawlega and the anonymous reviewers for their careful reading and helpful feedback.
\vspace{.7ex}

\input{prelims}

\input{problem}

\input{attacks}

\input{weak}

\input{non-poly}
\input{impact}

\input{open-questions}
%\input{conclusion}

%%%%%%%%%%%%%%%%%%%%%%%%%%%%%%%%%%%%%%%%

\bibliographystyle{plain}
\iffinal
\bibliography{springer-bib}
\else
\bibliography{bib}
\fi

\iffinal\else
\appendix
\input{proof}

\input{more-examples}

\input{impl}

\fi

\end{document}

%% file: intro.tex
\section{Introduction}

\noindent
With the advent of quantum computers, commonly deployed cryptosystems based on the integer-factorization or discrete-logarithm problems will need to be replaced by new post-quantum cryptosystems that rely on different assumptions.
Isogeny-based cryptography is a relatively new field within post-quantum cryptography.
An \emph{isogeny} is a non-zero rational map between elliptic curves that also preserves the group structure, and isogeny-based cryptography is based on the conjectured hardness of finding isogenies between elliptic curves over finite fields.

Isogeny-based cryptography stands out amongst post-quantum primitives 
due to the fact that isogeny-based key-exchange achieves the smallest key sizes of all candidates.
Isogeny-based schemes also appear to be fairly flexible;
for example, a relatively efficient 
post-quantum non-interactive key agreement protocol called CSIDH~\cite{10.1007/978-3-030-03332-3_15} is built on isogeny assumptions.

The \emph{Supersingular Isogeny Diffie--Hellman} protocol, or \emph{SIDH}, was the first practical isogeny-based key-exchange protocol, 
proposed in 2011 by Jao and De~Feo~\cite{jao2011towards}.
The security of SIDH relies on the hardness of solving (a special case of) the following problem:%
    \footnote{%
        See Section~\ref{sec:isogrep} for how the objects discussed are represented computationally.
    }

\begin{problem}[Supersingular Isogeny with Torsion (SSI-T)]\namedlabel{torsion}{SSI-T}\label{problemone}
    For a prime $p$ and
	smooth coprime integers $\None$ and $\Ntwo$, 
	given two supersingular elliptic curves $E_0/\F_{p^2}$ and $E/\F_{p^2}$
    connected by an unknown degree-$\None$ isogeny $\phi\colon E_0 \to E$,
	and given the restriction
    of~$\phi$
    to the $\Ntwo$-torsion of~$E_0$,
    recover an%
    \footnote{%
        \label{fn:isog_uniq}%
        These constraints do not necessarily uniquely determine $\phi$, 
        but any efficiently computable isogeny from $E_0$ to $E$ 
        is usually enough to recover the SIDH secret~\cite{galbraith2016security,takoisogeny}.
        Moreover, $\phi$ is unique whenever $B^2>4A$~\cite[\S\,4]{not-break-sidh}.
        % proof: Lemma V.1.2 and Corollary III.6.3 of Silverman
    }
    isogeny~$\phi$
    matching these constraints.
\end{problem}

\noindent
\ref{torsion} is a generalization of the \enquote{Computational Supersingular Isogeny problem}, or CSSI for short, defined in \cite{jao2011towards}.
Although the CSSI problem that appears in the literature also includes torsion information, we use the name \ref{torsion} to stress the importance of the additional torsion information. Additionally, we consider more flexibility in the parameters than CSSI to challenge the implicit assumption that even with torsion information the hardness of the protocol always scales with the degree of the isogenies and the characteristic $p$ of the field.

The best known way 
to break SIDH by treating it as a pure isogeny problem is a %meet-in-the-middle or
claw-finding
approach on the isogeny graph having classical complexity $\bigO(\hspace{-.2em}\sqrt\None\cdot\polylog(p))$
and no known quantum speedups viable in reality~\cite{DBLP:conf/crypto/JaquesS19}.%
\footnote{%
    Note that the naïve meet-in-the-middle approach has prohibitively
    large memory requirements. Collision finding
    à la van~Oorschot--Wiener
    thus
    performs better in practice, although its time complexity is worse in theory%
    ~\cite{DBLP:journals/iacr/AdjCCMR18}.
}
However, it is clear that \ref{torsion} provides the attacker with more information than
the \enquote{pure} supersingular isogeny problem, where the goal is to find an isogeny
between two given supersingular elliptic curves without any further hints or restrictions.

The first indication that additional torsion-point information
could be exploited to attack a supersingular isogeny-based cryptosystem 
was an active key-reuse attack against SIDH 
published in 2016~\cite{galbraith2016security} by Galbraith, Petit, Shani, and Ti.
In~\cite{galbraith2016security} the attacker sends key-exchange messages
with manipulated torsion points
and detects whether the key exchange succeeds.
This allows recovery of the secret key within $O({\log \None})$ queries.
To mitigate this attack, \cite{galbraith2016security}~proposes
using the Fujisaki--Okamoto transform,
which generically renders
a CPA-secure public-key encryption scheme CCA-secure,
and therefore thwarts those so-called \emph{reaction attacks}.
The resulting scheme 
\emph{Supersingular Isogeny Key Encapsulation},
or \emph{SIKE}~\cite{azarderakhsh2019supersingular} for short, is
the only isogeny-based submission to NIST's standardization project for post-quantum cryptography~\cite{nistpqc}, and is currently a Round 3 \enquote{Alternate Candidate}.

However,~\ref{torsion} can be easier than
finding isogenies in general.
Indeed, a line of work \cite{Petit2017,bottinellidark} revealed a 
separation between the hardness of the 
supersingular isogeny problem and \ref{torsion} for some parameterizations.
This is potentially concerning because
 several similar schemes have been proposed
that are based on the more general \ref{torsion}, and in particular, not clearly based on
the CSSI problem as stated in~\cite{jao2011towards} due to CSSI's restrictions on $A$ and $B$
\cite{DBLP:conf/asiacrypt/Costello20, sahusupersingular, DBLP:conf/isita/FurukawaKT18,DBLP:journals/iacr/AzarderakhshJJS19, boneh2020oblivious, de2019seta}. For example, for the security of the GSIDH $n$-party group key agreement \cite{DBLP:conf/isita/FurukawaKT18,DBLP:journals/iacr/AzarderakhshJJS19}, \ref{torsion} must hold for $B \approx A^{n-1}$.

A particular choice made in SIKE is to fix
the \enquote{starting curve} $E_0$ to be
a curve defined over $\F_p$ that has small-degree non-scalar endomorphisms; 
these are very rare properties within the set of all supersingular curves defined over~$\F_{p^2}$.
On its own, such a choice of starting curve does not seem to have any negative security implications for SIKE.
However, in addition to their active attack,~\cite{galbraith2016security} shows
that given an explicit description of \emph{both} curves' endomorphism rings,
it is (under reasonable heuristic assumptions)
possible to recover the secret isogeny in SIKE.
The argument in \cite{galbraith2016security} does not use torsion-point information, but only applies if the curves are sufficiently close; 
recently \cite{takoisogeny} showed that if torsion-point information is provided the two curves do not need to be close. 

The approach for solving \ref{torsion} introduced by Petit in
2017~\cite{Petit2017} exploits both torsion-point information 
and knowledge of the endomorphism ring of the special starting curve.
This attack is efficient for certain parameters,
for which the \enquote{pure} supersingular isogeny problem still appears
to be hard.
It uses the knowledge of the secret isogeny restricted to a large torsion subgroup to recover the isogeny itself,
giving a \emph{passive} heuristic polynomial-time attack on non-standard variants of SIDH
satisfying $\Ntwo>\None^4>p^4$.
However,
in practice, for all the SIDH-style schemes proposed in the literature so far,
both $\None$ and $\Ntwo$ are
taken to be divisors of $p^2-1$, allowing torsion points 
to be defined over small field extensions, which makes the resulting scheme more efficient.  
One of the contributions of this work is extending torsion-point attacks to have a stronger impact on parameterizations where A and B are divisors  of $p+1$ or $p^2-1$.

\subsection{Our contributions}

We improve upon and extend Petit's 2017 \emph{torsion-point attacks}~\cite{Petit2017} in several ways. 
Our technical results have the following cryptographic implications:
\begin{itemize}
    \item  We give an attack on $n$-party group key agreement \cite{DBLP:conf/isita/FurukawaKT18,DBLP:journals/iacr/AzarderakhshJJS19}, see Section~\ref{ss:frob_nonpoly} and in particular Table~\ref{table:group-key-exchange}.
    This attack applies to the GSIDH protocol of \cite{DBLP:conf/isita/FurukawaKT18}, 
    not to the SIBD procotol of \cite{DBLP:conf/isita/FurukawaKT18}.
    Our attack yields, under Heuristic~\ref{heur:deathattack-kate2}:
    \begin{itemize}
        \item A polynomial-time break for $n\geq 6$.
        \item An improved classical attack for $n\geq 5$.
        \item An improved quantum attack for $n\geq 3$ (compared to the asymptotic complexity for quantum claw-finding computed in~\cite{DBLP:conf/crypto/JaquesS19}).
    \end{itemize}
    
    We provide a Magma~\cite{bosma1997magma} implementation of our attack on $6$-party group key agreement, see \url{https://github.com/torsion-attacks-SIDH/6party}.
    \item We give an attack on B-SIDH~\cite{DBLP:conf/asiacrypt/Costello20} that, under Heuristic~\ref{heur:deathattack-kate}, 
    is asymptotically better than quantum claw-finding (with respect to~\cite{DBLP:conf/crypto/JaquesS19}), 
    although it does not weaken the security claims of~\cite{DBLP:conf/asiacrypt/Costello20} (see Section~\ref{impact-BSIDH}).
    \item We show that setting up a B-SIDH group key agreement in the natural way would yield a polynomial-time attack for 4 or more parties (see Section~\ref{sec:bsidhgroup}).
    \item More generally, we solve Problem~\ref{torsion} (under plausible explicit heuristics) in: 
    \begin{enumerate}
        \item Polynomial time when
        \begin{itemize}
            \item $j(E_0) = 1728$, $B > pA$, $p>A$, $A$ has (at most) $O(\log\log p)$ distinct prime factors, and $B$ is at most polynomial in $A$ (Proposition~\ref{prop:eq1} and Corollary~\ref{eqnofdeath}).
            \item $j(E_0) = 1728$, $B > \sqrt{p}A^2$, $p> A$, $A$ has (at most) $O(\log\log p)$ distinct prime factors, and $B$ is at most polynomial in $A$
            (Proposition~\ref{prop:eq2} and Corollary~\ref{eqnoflessdeath}).
            \item $E_0$ is a specially constructed \enquote{backdoor curve}, $B>A^2$, and $A$ has (at most)
            $O(\log\log p)$ distinct prime factors
            (Theorem~\ref{trapdoor} and Algorithm~\ref{alg:insecure}).
            \item $j(E_0) = 1728$ and $p$ is a specially constructed backdoor prime
            (Sections~\ref{ss:Insecure p} and~\ref{subsec:param}).
        \end{itemize}
        \item Superpolynomial time but asymptotically more efficient than meet-in-the-middle on a classical computer when
        \begin{itemize}
            \item $j(E_0) = 1728$, $B > \max\left\{ \sqrt{p}A^{\frac34}, A, p\right\}$, $A$ has (at most) $O(\log\log p)$ distinct prime factors, and $B$ is at most polynomial in $A$ (Corollary~\ref{cor:nonpolydual}).
            \item $j(E_0) = 1728$, $B > \sqrt{p}A$, $A$ has (at most) $O(\log\log p)$ distinct prime factors, and $B$ is at most polynomial in $A$ (Corollary~\ref{cor:nonpolyfrob}). 
            \item $E_0$ is a specially constructed \enquote{backdoor curve} and $A$ has (at most)
            $O(\log\log p)$ distinct prime factors (Proposition~\ref{extend-improve}).
        \end{itemize}
        \item Superpolynomial time but asymptotically more efficient than quantum claw-finding (with respect to~\cite{DBLP:conf/crypto/JaquesS19}) when
            $j(E_0) = 1728$, $B > \sqrt{p}$, $A$ has (at most) $O(\log\log p)$ distinct prime factors, and $B$ is at most polynomial in $A$ (Corollary~\ref{cor:nonpolyfrob}).
    \end{enumerate}
 \end{itemize}
 
\input{newpicture.tex}

\vspace{-5ex}
\noindent
These cryptographic implications are consequences of the following new mathematical results:
\begin{itemize}
    \item In Section~\ref{sec:prob}, we formalize the hardness assumption and reduction implicit in~\cite{Petit2017}. We call this hardness assumption the \textit{Shifted Lollipop Endomorphism} (SLE) Problem.
    \item In Section~\ref{sec:improvements}, we give two improved reductions to $\SLE$ (leading to our \emph{dual isogeny attack} and \emph{Frobenius isogeny attack}).
    \item In Section~\ref{sec:insecure}, we:
    \begin{itemize}
        \item Introduce \enquote{backdoor} curves, which, when used as $E_0$, allows us to solve \ref{torsion} in polynomial time if $B > A^2$. 
        \item Give a method to construct backdoor curves and study their frequency.
        \item Introduce \enquote{backdoor} primes, which, when used for $p$, allows us to solve \ref{torsion} in polynomial time.
    \end{itemize}
    \item In Section~\ref{sec:non-poly}, we show how to extend both the dual isogeny attack and the Frobenius isogeny attack to allow for superpolynomial attacks.
 \end{itemize}   
 
We emphasize that none of our attacks apply to the NIST candidate SIKE: for each attack described in this paper, at least one aspect of SIKE needs to be changed (e.g., the balance of the degrees of the secret isogenies, the starting curve, or the base-field prime). 

\CP{Comment on memory costs and improve theorems to say they are zero?}

\subsection{Comparison to earlier work}\label{sec:earlier-work}

In \cite{DBLP:journals/iacr/AzarderakhshJJS19}, the authors estimated that the attack from \cite{Petit2017}
would render their scheme insecure for 400 parties or more. In contrast, we give a complete break when there are at least 6 parties.

The cryptanalysis done by Bottinelli et al.~\cite{bottinellidark} also gave a reduction in the same vein as Petit's 2017 paper~\cite{Petit2017}.
Our work overlaps with theirs (only) in Corollary~\ref{eqnoflessdeath},
and the only similarity in techniques is in the use of \enquote{triangular decomposition}~\cite[\S\,5.1]{bottinellidark}, see the middle diagram in Figure~\ref{FigDiamond}.
Although their improvement is akin to the one given by our dual isogeny attack, they require additional
(shifted lollipop) endomorphisms\CV{we do use the dual isogeny, it is more that we need extra endomorphisms};
unfortunately, we have not found a way to combine the two methods. 
Moreover, our results go beyond~\cite{bottinellidark} in several ways: we additionally introduce the Frobenius isogeny attack (in particular giving rise to our attack on group key agreement).
We consider multiple trade-offs for both the dual and the Frobenius isogeny attacks by allowing for superpolynomial attacks, as well as considering other starting curves and base-field primes.

\subsection{Outline}

In Section~\ref{sec:Pre} we go over various preliminaries, including reviewing SIDH.
In Section~\ref{sec:prob} we define the relevant hard isogeny problems and give a technical preview;
we also outline the idea behind our attacks and how they give rise to reductions of the~\ref{torsion} Problem.
In Section~\ref{sec:improvements}
we prove our reductions and give two new algorithms to solve~\ref{torsion} in polynomial time for certain parameter sets.
In Section~\ref{sec:insecure} we introduce backdoor curves $E_0$ and backdoor primes $p$ for which we can solve~\ref{torsion} in polynomial time for certain parameter sets.
In Section~\ref{sec:non-poly} we extend the attacks of Sections~\ref{sec:improvements} and~\ref{sec:insecure} to superpolynomial attacks.
In Section~\ref{sec:impact} we give the impact of our attacks on cryptographic protocols in the literature.
In Section~\ref{sec:open} we pose an open question on constructing new reductions.

%% file: newpicture.tex
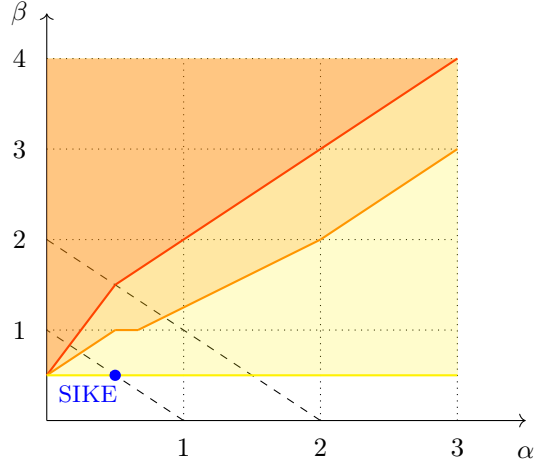
\begin{figure}
    \vspace{-4ex}
    \centering
    \begin{tikzpicture}[yscale=1.2,xscale=1.8]
    %\tikzstyle{every node}=[dot]
    \path[->] (0,0) edge[black] (3.5,0) ;
    \path[->] (0,0) edge[black] (0,4.5) ;
    \draw[thin, dotted] (0,1) -- (3,1);
    \draw[thin, dotted] (0,2) -- (3,2);
    \draw[thin, dotted] (0,3) -- (3,3);
    \draw[thin, dotted] (0,4) -- (3,4);
    \draw[thin, dotted] (1,0) -- (1,4);
    \draw[thin, dotted] (2,0) -- (2,4);
    \draw[thin, dotted] (3,0) -- (3,4);
    \node at (1,-0.3) {1} ;
    \node at (2,-0.3) {2} ;
    \node at (3,-0.3) {3} ;
    \node at (3.5,-0.35) {$\alpha$} ;
    \node at (-0.2,1) {1} ;
    \node at (-0.2,2) {2} ;
    \node at (-0.2,3) {3} ;
    \node at (-0.2,4) {4} ;
    \node at (-0.2,4.5) {$\beta$} ;

    % lines corresponding to typical sidh/ group key exchange, and bsidh setting
    \draw[thin,dashed] (2,0) -- (0,2) ;
    \draw[thin,dashed] (1,0) -- (0,1);

    % polynomial-time attacks
    \draw[red,thick] (0,1/2) -- (1/2,3/2); % beta>1/2+2\alpha
    \draw[red,thick] (1/2,3/2) -- (3,4); % \beta>\alpha+1
    \draw[fill=red,draw=none,fill opacity=.2] (0,1/2)--(1/2,3/2)--(3,4)--(0,4)--cycle;

    % classical attack improvements
    \draw[orange,thick] (0,1/2) -- (1/2,1) ; %\beta> 1/2+ \alpha
    \draw[orange,thick] (1/2,1) -- (2/3,1) ; %\btea> 1
    \draw[orange,thick] (2/3,1) -- (2,2) ; %\btea> 3/4 \alpha+1/2
    \draw[orange,thick] (2,2) -- (3,3) ; % \beta>\alpha
    \draw[fill=orange,draw=none,fill opacity=.2] (0,1/2)--(1/2,1)--(2/3,1)--(2,2)--(3,3)--(3,4)--(0,4)--cycle;
    %

     % quantum attack improvements
    \draw[yellow,thick] (0,1/2) -- (3,1/2) ; %\beta>1/2
    \draw[fill=yellow,draw=none,fill opacity=.2] (0,1/2)--(3,1/2)--(3,4)--(0,4)--cycle;
    % N1N2>p if \alpha <1/2

    \node[fill=blue,circle,inner sep=1.5pt] at (1/2,1/2) {};
    \node[blue] at (0.3,0.3) {\smaller SIKE};

    \end{tikzpicture}
    \vspace{-1ex}
    \caption{Performance of our %torsion-point 
    attacks for $j(E_0)=1728$.
    %(and other curves with similar properties).
    Here $A\approx p^{\alpha}$ and $B\approx p^{\beta}$.
    Parameters above the red, orange and yellow curves are parameters admitting a polynomial-time attack,
    an improvement over the best classical attacks,
    and an improvement over the best quantum attacks respectively.
    Parameters below the upper dashed line are those allowing
    $AB\mid(p^2-1)$ as in~\cite{DBLP:conf/asiacrypt/Costello20}.
    Parameters below the lower dashed line are those allowing $AB\mid(p+1)$
    as in~\cite{azarderakhsh2017supersingular,azarderakhsh2019supersingular}. The blue dot corresponds to SIKE parameters.
    \label{fig:attackperf}}
\end{figure}

%% file: prelims.tex
\section{Preliminaries}\label{sec:Pre}

% outline?

\subsection{The Supersingular Isogeny Diffie--Hellman protocol family}\label{subsec:sidh}

We give a somewhat generalized high-level description of SIDH~\cite{jao2011towards}.
Recall that $E[N]$ denotes the $N$-torsion subgroup of an elliptic curve $E$
and $[m]$ denotes scalar multiplication by $m$.
The public parameters of the system are two smooth coprime numbers $\None$ and $\Ntwo$, a prime $p$ of the form $p=\None\Ntwo f-1$, where $f$ is a small cofactor, and a supersingular elliptic curve $E_0$
defined over $\F_{p^2}$ together with points $P_A,Q_A,P_B,Q_B\in E_0$ such that
$E_0[\None]=\langle P_A,Q_A\rangle$
and
$E_0[\Ntwo]=\langle P_B,Q_B\rangle$.

\vspace{.4ex}
\noindent
The protocol then proceeds as follows:
\newcommand\GGAA{G_{\hspace{-.1em}A}}
\newcommand\GGBB{G_{\hspace{-.1em}B}}
\begin{enumerate}
    \item Alice chooses a random cyclic subgroup of $E_0[\None]$ as $\GGAA=\langle P_A + [x_A] Q_A\rangle$ and Bob chooses a random cyclic subgroup of $E_0[\Ntwo]$ as $\GGBB=\langle P_B + [x_B] Q_B\rangle$.
    \item Alice computes the isogeny $\phi_A: E_0\rightarrow E_0/\langle \GGAA \rangle=:E_A$ and Bob computes the isogeny $\phi_B: E_0\rightarrow E_0/\langle \GGBB \rangle=:E_B.$
    \item Alice sends the curve $E_A$ and the two points $\phi_A(P_B),\phi_A(Q_B)$ to Bob.
        Similarly, Bob sends $\big(E_B,\phi_B(P_A),\phi_B(Q_A)\big)$ to Alice.
    \item Alice and Bob use the given torsion points to obtain the shared secret curve $E_0/\langle \GGAA,\GGBB \rangle$. To do so, Alice computes $\phi_B(\GGAA)=\phi_B(P_A) + [x_A] \phi_B(Q_A)$ and uses the fact that $E_0/\langle \GGAA,\GGBB \rangle \cong E_B/\langle \phi_B(\GGAA)\rangle$. Bob proceeds analogously.

\erase{
        \vspace{.4ex}

        {\smaller\smaller
        (Publishing the action of the
        secret isogeny on public points
        can be considered
        \emph{the}
        core idea behind SIDH:
        Alice needs $\phi_B(\GGAA)$
        to complete the key exchange,
        but
        Bob must keep $\phi_B$ secret
        and Alice must keep $\GGAA$ secret.
        Handing out the action of $\phi_B$
        on a
        publicly known
        group that \emph{contains}
        the secret
        $\GGAA$ is a clever workaround for this problem.)
        \par}
}
\end{enumerate}

The SIKE proposal~\cite{azarderakhsh2019supersingular}
suggests various choices of $(p,\None,\Ntwo)$ depending on the targeted security level:
All parameter sets use powers of two and three for~$\None$ and~$\Ntwo$,
respectively, with $\None \approx \Ntwo$ and $f=1$.
For example, the smallest parameter set suggested in~\cite{azarderakhsh2019supersingular}
uses $p = 2^{216} \cdot 3^{137}-1$.
Other constructions belonging to the SIDH \enquote{family tree} of protocols use different types of parameters~\cite{DBLP:conf/isita/FurukawaKT18,DBLP:conf/asiacrypt/Costello20,DBLP:journals/iacr/AzarderakhshJJS19,sahusupersingular}.

We may assume knowledge of $\End(E_0)$:
The only known way to construct supersingular elliptic curves is by
reduction of elliptic curves
with CM by a small discriminant
(which implies small-degree endomorphisms: see~\cite{love2019supersingular,rational-irrational}),
or by isogeny walks starting from such curves
(where knowledge of the path reveals the endomorphism ring,
thus requiring trusted setup).
A common choice when $p\equiv3\pmod4$ is $j(E_0)=1728$
or a small-degree isogeny neighbour of that curve~\cite{azarderakhsh2019supersingular}. 
Various variants of SIDH exist in the literature. 
We will call a variant an \textit{SIDH-like} protocol if its security
can be broken
by solving \ref{torsion} for some values of $A$ and $B$.

In \cite{DBLP:journals/iacr/AzarderakhshJJS19} the authors propose the following $n$-party key agreement, first introduced as GSIDH in~\cite{DBLP:conf/isita/FurukawaKT18}.\footnote{%
\cite{DBLP:conf/isita/FurukawaKT18} also proposes a different group key agreement, SIBD,
to which our attack does not apply.}
The idea is to use primes of the form $p=f\prod_{i=1}^n\ell_i^{e_i}-1$ where $\ell_i$ is the $i$-th prime number, 
the $i$-th party's secret isogeny has degree $\ell_i^{e_i}$, 
the $i$-th participant provides the images of a basis of the $\prod_{j=1}^n\ell_j^{e_j}/\ell_i^{e_i}$ torsion,
and $f$ is a small cofactor.
They choose the starting curve to be of \jinvariant 1728 and choose the $e_i$ in such a way that all the 
$\ell_i^{e_i}$ are of roughly the same size. 
This is an example of an SIDH-like protocol; 
for this protocol to be secure it is required that
\ref{torsion} be hard when 
$A=\ell_1^{e_1}$ and $B=f\prod_{i=2}^{n}\ell_i^{e_i}$.
However, we prove in Theorem~\ref{thm:polybreakgroupkey} that \ref{torsion} can be solved in polynomial time for 6 or more parties; also see Table~\ref{table:group-key-exchange} for the complexity of our attack for any number of parties.

Another example of a SIDH-like scheme is B-SIDH \cite{DBLP:conf/asiacrypt/Costello20}. 
In B-SIDH, the prime has the property that $p^2-1$ is smooth (as opposed to just $p+1$ being smooth)
and $A\approx B\approx p$. 
It would seem that choosing parameters this way one has to work over $\mathbb{F}_{p^4}$ 
but in fact the scheme simultaneously works with the curve and its quadratic twist 
(i.e., a curve which is not isomorphic to the original curve over $\mathbb{F}_{p^2}$ but has the same \jinvariant) 
and avoids the use of extension fields. 
The main advantage of B-SIDH is that the base-field primes used can be considerably smaller than the primes used in SIDH. 
We discuss the impact of our attacks of B-SIDH in Subsection~\ref{impact-BSIDH};
although we give an improvement on the quantum attack of~\cite{DBLP:conf/crypto/JaquesS19}
the parameter choices in~\cite{DBLP:conf/asiacrypt/Costello20} are not affected as they were chosen with a significant quantum security margin.
\LP{should mention \cite{takoisogeny} here?}

The general concept of using primes of this form extends beyond the actual B-SIDH scheme. 
As a final example of an SIDH-like scheme, consider the natural idea of using B-SIDH in a group key agreement context. 
The reason that this construction is a natural choice is that a large number of parties implies a large base-field prime, 
which is an issue both in terms of efficiency and key size. 
Using a B-SIDH prime could in theory enable the use of primes of half the size.
However, as we show in Corollary~\ref{cor:bsidhgroupkey}, such a scheme is especially susceptible to our attacks and is broken in polynomial time for 4 or more parties.

\subsection{Notation}

Throughout this paper,
we work with the field $\F_{p^2}$ for a prime $p$.
In our analysis we often want to omit factors polynomial in $\log p$;
 as such, from this point on we will abbreviate $\bigO(g\cdot\polylog(p))$ by $\Ostar(g)$.%
\footnote{%
    Each occurrence of
    $\polylog(p)$
    is shorthand
    for a concrete,
    fixed polynomial in $\log p$.
    (The notation is \emph{not} meant to imply that
    all instances of $\polylog(p)$
    be the same.)
}
Similarly, a number is called
\textit{smooth}, without further qualification, if all of its
prime factors are $\Ostar(1)$.
\textit{Polynomial time} without explicitly mentioning the variables means \enquote{polynomial in the representation size of the input}\,---\,usually the logarithms of integers.
An algorithm is called \textit{efficient} if its runs in polynomial time.

We let $\BB_{p,\infty}$ denote the quaternion algebra %over $\Q$
ramified at~$p$ and~$\infty$, for which we use a fixed $\Q$-basis $\langle 1,\i,\j,\ij\rangle$ such that $\j^2=-p$ and $\i$ is a nonzero endomorphism of minimal norm satisfying $\ij=-\ji$. Quaternions are treated symbolically throughout; they are simply formal linear combinations of $1,\i,\j,\ij$.

For any positive integer $N$ we write $\operatorname{sqfr}(N)$ for the squarefree part of $N$.

\subsubsection{Representation of elliptic-curve points and isogenies.}
\label{sec:isogrep}
We will generally require that the objects we are working with
have \enquote{compact} representation (that is, size $\polylog(p)$ bits),
and that maps can be evaluated at points
of representation size $\polylog(p)$
in time $\polylog(p)$.

In the interest of generality,
we will not force a specific choice
of representation,
but for concreteness,
the following data formats
are examples of suitable instantiations:
\begin{itemize}[itemsep=1ex]
    \item
            For an elliptic curve $E$
            defined over an extension of $\F_p$
            and an integer $N$,
            a point in $E[N]$ may be stored as a tuple
            consisting of one point in $E[q_i^{e_i}]$
            for each prime power $q_i^{e_i}$
            in the factorization of $N$,
            each represented naïvely as coordinates.
            This \enquote{CRT-style} representation
            has size $\polylog(p)$
            when $N$ is powersmooth and
            polynomial in $p$.
            (In some cases, storing points in $E[N]$ naïvely
            may be more efficient,
            for instance in the beneficial situation that
            $E[N]\subseteq E(\F_{p^k})$
            for some small extension degree $k$.)
    \item
            A smooth-degree isogeny may be represented
            as a sequence
            (often of length one)
            of isogenies, each of which
            is represented by an
            (often singleton)
            set of generators
            of its kernel subgroup.
    \item
            Endomorphisms of a curve $E_0$ with known
            endomorphism ring spanned by a set of
            efficiently evaluatable endomorphisms
            may be stored as a formal $\Z$-linear
            combination of such \enquote{nice} endomorphisms.
            Evaluation is done by first evaluating
            each basis endomorphism separately,
            then taking the appropriate linear combination
            of the resulting points.
\end{itemize}

In some of our algorithms,
we will deal with the \textit{restriction} of an isogeny
to some $N$-torsion subgroup,
where $N$ is smooth.
This object is motivated by
the auxiliary points $\phi_A(P_B),\phi_A(Q_B)$
given in the SIDH protocol (Section~\ref{subsec:sidh}),
and it can be represented in the same way:
The restriction of an isogeny $\phi\colon E\to E'$
to the $N$-torsion subgroup $E[N]$
is stored as a tuple of points
$(P,Q,\phi(P),\phi(Q))\in E^2\times E'^2$,
where $\{P,Q\}$ forms a basis of $E[N]$.
Then,
to evaluate $\phi$ on any other $N$-torsion point $R\in E[N]$,
we first decompose $R$ over the basis $\{P,Q\}$, yielding a linear combination $R=[i]P+[j]Q$.
(This two-dimensional discrete-logarithm computation is feasible in polynomial time
as $N$ was assumed to be smooth.)
Then, we may simply recover $\phi(R)$ as $[i]\phi(P)+[j]\phi(Q)$,
exploiting the fact that $\phi$ is a group homomorphism.

\subsection{Quantum computation cost  assumptions\label{sec:prel:quantumcost}}

In the context of NIST's post-quantum cryptography standardization process~\cite{nistpqc}, there is a significant ongoing effort to estimate the quantum cost of fundamental cryptanalysis tasks in practice.
In particular, while it seems well-accepted that Grover's algorithm provides a square-root quantum speedup, the complexity of the claimed cube-root claw-finding algorithm of Tani~\cite{tani} has been disputed by Jaques and Schanck~\cite{DBLP:conf/crypto/JaquesS19}, and the topic is still subject to ongoing research~\cite{DBLP:journals/iacr/JaquesS20}.

Several attacks we present in this paper use claw-finding algorithms as a subroutine, and the state-of-the-art algorithms against which we compare them are also claw-finding algorithms.
We stress, however, that the insight provided by our attacks is independent of the choice of the quantum computation model.
For concreteness we chose the RAM model studied in detail by Jaques and Schanck in~\cite{DBLP:conf/crypto/JaquesS19}, in which it is argued that quantum computers do not seem to offer a significant speedup over classical computers for the task of claw-finding.
 Adapting our various calculations to other existing and future quantum computing cost models, in particular with respect to claw-finding, is certainly possible.

%% file: problem.tex
\section{Overview}\label{sec:prob}

Standard attacks on SIDH follow two general approaches:
they either solve the supersingular isogeny problem directly,
or they reduce finding an isogeny to computing endomorphism rings.
However, SIDH is based on \ref{torsion} introduced above, where an adversary 
is also given the restriction of the secret isogeny
to the $\Ntwo$-torsion of the starting curve $E_0$. 
Exploiting this $B$-torsion information led to a new line of attack as first illustrated in~\cite{Petit2017}.

In Subsection~\ref{ss:hard_pro} we 
discuss the \ref{prob:iso} and \ref{torsion}.
Petit's work was the first to show an apparent separation 
between the hardness of \ref{torsion} and the hardness of the \ref{prob:iso} 
in certain settings.
In this work we introduce a new isogeny problem, the Shifted Lollipop Endomorphism Problem (SLE). This problem was implicit in Petit’s work~\cite{Petit2017}, which contained a purely algebraic reduction from~\ref{torsion} to this new hard problem.
We improve upon the work of~\cite{Petit2017} by giving two significantly stronger reductions.
In Subsection~\ref{subsec:petit} we sketch the main idea behind the reduction obtained by Petit.
In Subsection~\ref{subsec:techprev} we present a technical overview which covers the ideas behind our two improved reduction variants.

In Section~\ref{sec:improvements}
we will present and analyze our two reductions, and 
give algorithms to solve $\SLE$ for certain parameter sets.
As we will see, the combination of our reductions and our algorithms
to solve particular parameter sets of $\SLE$ give rise to 
two families of improvements on the torsion-point attacks of~\cite{Petit2017} on SIDH-like protocols;
these attacks will additionally exploit the dual of the secret isogeny and the Frobenius isogeny.

\subsection{Hard isogeny problems} \label{ss:hard_pro}

We first review the most basic hardness assumption in isogeny-based cryptography:

\begin{problem}[Supersingular Isogeny]\namedlabel{prob:iso}{Supersingular Isogeny Problem}
    Given a prime $p$,
	a smooth integer $A$, 
	and two supersingular elliptic curves $E_0/\F_{p^2}$ and $E/\F_{p^2}$
    guaranteed to be $A$-isogenous,
    find an isogeny~$\phi\colon E_0\to E$ of degree~$A$.%
\end{problem}

In SIDH, we denote Alice's secret isogeny $\phi_A:E_0\rightarrow E_A$,
but in general we will denote some unknown isogeny by $\phi:E_0\rightarrow E.$

Recall that Alice's public key contains not only the curve $E$ 
but also the points $\phi(P),$  $\phi(Q)$ for a fixed 
basis $\{P,Q\}$ of $E_0[B]$. 
Since $B$ is smooth, knowing $\phi(P)$ and $\phi(Q)$ allows us to efficiently compute 
the restriction of $\phi$ to the torsion subgroup $E_0[B]$~\cite{pohlig1978improved}.
Hence, it is more accurate to say that the security of SIDH is based on~\ref{torsion}, which includes this additional torsion information.

One additional fact that is often overlooked is that the hardness of SIDH is
not based on a \textit{random} instance of \ref{torsion}, because the 
starting curve is fixed and has a well-known endomorphism ring
with small degree endomorphisms.
It is known that given an explicit description of both endomorphism rings $\End(E)$ and $\End(E_0)$,
it is (under reasonable heuristic assumptions)
possible to recover the secret isogeny~\cite{galbraith2016security,takoisogeny}.
However, it is not clear if knowing only one of $\End(E)$ and $\End(E_0)$ 
makes the isogeny problem easier.

Petit was the first to observe that knowing $\End(E_0)$
could be useful to show an apparent separation between
the hardness of the~\ref{prob:iso} and the hardness of \ref{torsion}.
In particular, in~\cite{Petit2017} Petit gave a reduction from \ref{torsion} 
to the following problem,
which we will call the \textit{Shifted Lollipop Endomorphism} (SLE) Problem, where $N=B.$

\begin{problem}[Shifted Lollipop Endomorphism (SLE$_{N,\lambda}$)]\namedlabel{prob:lolli}{SLE}
Let $p$ be a prime,  $\None$ and $\Ntwo$ be
smooth coprime integers, and a supersingular elliptic curve $E_0/\F_{p^2}.$
Given a positive integer~$N$, find
the restriction of a trace-zero endomorphism $\theta\in \End(E_0)$ to $E_0[\Ntwo]$,
an integer $d$ coprime to $\Ntwo$, 
and a smooth integer $0< e < \lambda$ such that
\begin{equation}\label{eq:lolli}
\None^2 \deg\theta + d^2=Ne. 
\end{equation}
When $\lambda$ is left unspecified 
we let SLE$_{N}$ denote SLE$_{N,\Ostar(1)}$.
\end{problem}

Notice that~\ref{prob:lolli}$_N$ only depends on
the parameters $(p, \None, \Ntwo, E_0)$.
It does not depend on an unknown isogeny (it depends on $A$, which in practice will be the degree of the unknown isogeny). 
Thus solving \ref{prob:lolli}$_N$ can be completed in a precomputation phase and applied to any unknown isogeny in a fixed SIDH protocol.
In~\cite{Petit2017}, Petit was able to show solutions to \ref{prob:lolli}$_N$ where $N=B$
in certain cases, where $\End(E_0)$ was known
and has small-degree, non-scalar endomorphisms.

The goal of this work is to further investigate for which parameters
there exists
a separation between \ref{torsion} and the  \ref{prob:iso}.
Intuitively, \ref{prob:lolli}$_N$ should become easier to solve as $N$ increases,
however, this is not true in general and it is unclear how to find efficient reductions to \ref{prob:lolli}$_N$
for most values of $N$. 
To this end, we will give two reductions: one reduction from \ref{torsion}
to \ref{prob:lolli}$_{N,\lambda}$ where $N=B^2$, 
and the other where
$N=B^2p$. Both reductions run in  $\Ostar(\lambda^{\frac{1}{2}})$,
assuming $A$  has only $O(\log\log p)$ distinct prime factors,
see Theorems~\ref{more efficient} and~\ref{frob-thm-kate}. 
We then investigate their impact on supersingular isogeny-based protocols.

\subsection{Petit's torsion-point attack}\label{subsec:petit}

We begin this subsection by sketching Petit's reduction 
 from \ref{torsion} 
to \ref{prob:lolli}$_N$ where $N=B$.
Suppose we are given an instance of SSI-T,
that is, $(p,\None,\Ntwo, E_0,E, \phi\rvert_{E[\Ntwo]})$,
where the goal is to recover the unknown isogeny $\phi$.
We call an endomorphism on $E$ that has the form $\phi \circ \theta \circ \widehat{\phi}$ for some endomorphism $\theta$ on $E_0$ 
a \textbf{lollipop endomorphism}, and an endomorphism of the form $\phi \circ \theta \circ \widehat{\phi} + [d]$ for $d \in \Z$ a \textbf{shifted lollipop endomorphism};
see Figure \ref{FigLolli}
(this is the motivation for the name of Problem~\ref{prob:lolli}).
We will now discuss how to find a shifted lollipop endomorphism, as
we will show in Lemma~\ref{lem:shift_lolli} how to use the resulting shifted lollipop endomorphism 
to recover the secret isogeny.

\begin{figure}[ht]
\begin{adjustbox}{max totalsize={.8\textwidth}{.15\textheight},center}
\begin{tikzcd}[column sep=normal]
E_0 \arrow[dd,shift right, red, swap,"\phi"] \ar[loop left,<-,blue, out=130,
  in=50, distance=1cm,"\theta"] \\ 
  \\
E  \arrow[uu,shift right, red, swap, "\widehat{\phi}"]
\end{tikzcd}
\hspace{1cm}
\begin{tikzcd}[column sep=normal]
\\ \\ \\
 E
\arrow[
  out=60,
  in=100,
  loop,
  distance=3cm, "\phi \circ \theta \circ \widehat{\phi}",swap,blue ] 
  \arrow[
  out=80,
  in=120,
  loop,
  distance=3cm,"\phi \circ \theta \circ \widehat{\phi} + \bd",swap ,blue]
\end{tikzcd}
\end{adjustbox}
\caption{Lollipop and Shifted Lollipop endomorphisms. 
The name \enquote{lollipop} endomorphism was inspired by the diagram on the left.}
\label{FigLolli}
\end{figure}
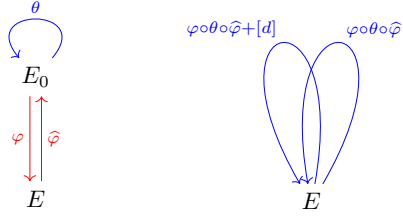

The main idea of Petit's original attack is that if
$(\theta, d,e)$ forms a solution to \ref{prob:lolli}$_B$,
then  $\tau=\phi\circ\theta \circ \hat{\phi} +[d]$ is a 
shifted lollipop endomorphism of degree $Be$ where $e$ is smooth. 
Since $\deg \tau = \Ntwo e$, it follows that $\tau$ also decomposes as $\tau = \eta \circ \phi$ for two isogenies 
$\phi:E\rightarrow E_1$ and $\eta:E_1\rightarrow E$ of 
degrees $B$ and $e$; see Figure \ref{FigOrigPetit}.

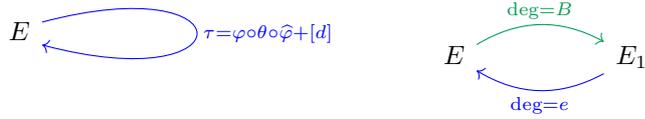
\begin{figure}[ht]
\begin{adjustbox}{max totalsize={.8\textwidth}{.2\textheight},center}
\begin{tikzcd}[row sep=2.5em]
E \arrow[loop right,  blue,distance=8em, start anchor={[yshift=1ex]east}, end anchor={[yshift=-1ex]east}]{}{\tau = \phi \circ \theta \circ \widehat{\phi}+[d]} &&\\
\end{tikzcd}
\hspace{1 cm}
\begin{tikzcd}[row sep=2.5em]
E  \arrow[rr, bend left, ForestGreen,"\deg=B"]  && E_{1}  \arrow[ll, bend left, blue,"\deg=e"]
\end{tikzcd}
\end{adjustbox}
\caption{A decomposition of $\tau$ in Petit's original attack}\label{FigOrigPetit}
\end{figure}

The restriction of $\phi$ to $E_0[\Ntwo]$ given in Alice's public key
can be used to construct the $B$\nobreakdash-isogeny in the decomposition 
(the green arrow in Figure \ref{FigOrigPetit}), see~\cite{Petit2017} for details.
This can be done efficiently if $\theta$ is in a representation 
that can be efficiently evaluated on $E_0[\Ntwo]$.
As $e$ is smooth, the $e$-isogeny
in the decomposition (the blue arrow)
can be found via brute-force in time $\Ostar(e^{\frac{1}{2}})$. This gives us $\tau$.
Subtracting $\bd$ from $\tau$ gives $ \phi \circ \theta \circ \widehat{\phi}$.

Suppose the lollipop endomorphism $\rho= \phi \circ \theta \circ \widehat{\phi}$ 
is cyclic.  Then $\ker(\rho)\cap E_1[A]=\ker\widehat{\phi}$.
(The kernel of $\rho$ can be calculated as $A$ is smooth.)
Once we have found $\widehat{\phi},$ it is easy to find 
the unknown isogeny $\phi.$
If $\rho$ is not cyclic, then one can still recover $\phi$ if  $A$  has $O(\log\log p)$ distinct prime factors by using a technical approach developed in \cite[Section 4.3]{Petit2017}, for further details see Lemma \ref{lem:shift_lolli}. Thus we have a reduction from \ref{torsion}
to \ref{prob:lolli}$_N$ where $N=B$, which is formalized in the following theorem.

\begin{theorem}\label{thm:Petit_original}
Suppose we are given an instance of~\ref{torsion} where $A$ has $O(\log\log p)$ distinct prime factors.
Assume we are given
the restriction of  a trace-zero endomorphism $\theta\in \End(E_0)$
to $E_0[B]$,
an integer $d$ coprime to $B$, 
and a smooth integer $e$ such that
$$\deg(\phi\circ\theta\circ\hat{\phi}+[d])=Be. $$
Then we can compute $\phi$ in time 
    $\Ostar(\hspace{-.2em}\sqrt{e})=\bigO(\hspace{-.2em}\sqrt{e}\cdot\polylog(p))$.
\end{theorem}

\subsection{Technical preview}\label{subsec:techprev}

Although the attack of~\cite{Petit2017} was the first to establish an apparent
separation between the hardness of \ref{torsion}
and the hardness of supersingular isogeny problem,
it did not affect the security of any  cryptosystems that appear in the literature.
In this paper, we give two attacks improving upon~\cite{Petit2017} by additionally exploiting the dual and the Frobenius conjugate of the secret isogeny respectively.

The first attack, which we call the \textbf{dual isogeny attack},
corresponds to reducing \ref{torsion}
to \ref{prob:lolli}$_N$ where $N=B^2$.\footnote{
See also \cite{bottinellidark} for a different reduction to \ref{prob:lolli}$_{B^2}$,
cf. Subsection~\ref{sec:earlier-work}.
}
The second attack, which we call the \textbf{Frobenius isogeny attack},
corresponds to reducing \ref{torsion}
to \ref{prob:lolli}$_N$ where $N=B^2p$.
The run-time of each attack depends on the parametrization of the cryptosystem, 
and one may perform better than the other for some choices of parameters. 
We show the details in Theorem~\ref{more efficient} and Theorem~\ref{frob-thm-kate}.
We begin by sketching the main ideas behind the reductions.

In the dual isogeny attack,
finding a solution $(\theta,d,e)$ to \ref{prob:lolli}$_N$ with $N=B^2$
corresponds to finding a shifted lollipop endomorphism 
$\tau = \phi \circ \theta \circ \widehat{\phi} +[d]$ on $E$ of degree  $ B^2e$, with $e$ smooth.
Assume $\tau$ is cyclic (only for simplicity in this overview; the general case is Theorem~\ref{more efficient}). Then
since $\deg \tau = \Ntwo^2 e$, it follows that $\tau$ also decomposes as $\tau = \phi'\circ \eta \circ \phi$ for three isogenies 
$\phi, \eta$ and $\phi'$ of 
degrees $B, e$ and $B$, respectively: see the middle diagram in Figure \ref{FigDiamond}. \CP{I am not sure how much insight is provided by this paragraph? I mean, when I wrote the first paper I had or stated as an open problem to replace $B$ by $B^2$ but crucially did not know how to do it; the paragraph says nothing about this gap and only focuses on trivial stuff}

In the Frobenius isogeny attack,
finding a solution $(\theta,d,e)$ to  \ref{prob:lolli}$_N$ with $N=B^2p$
corresponds to finding a shifted lollipop endomorphism $\tau= \phi \circ \theta \circ \widehat{\phi} +[d]$ that has degree  $B^2pe$, with $e$ smooth.
Assume $\tau$ is cyclic (only for simplicity in this overview; the general case is Theorem~\ref{frob-thm-kate}). 
Since $\deg \tau = \Ntwo^2 p e$, it follows that $\tau$ also decomposes as $\tau = \phi'\circ \eta \circ \pi \circ \phi$ for four isogenies 
$\phi, \pi, \eta$ and $\phi'$ of 
degrees $B, p, e$ and $B$, respectively,
where the isogeny of degree $p$ is the Frobenius map $(x,y) \rightarrow (x^p,y^p)$: see the right-hand diagram in Figure \ref{FigDiamond}. 

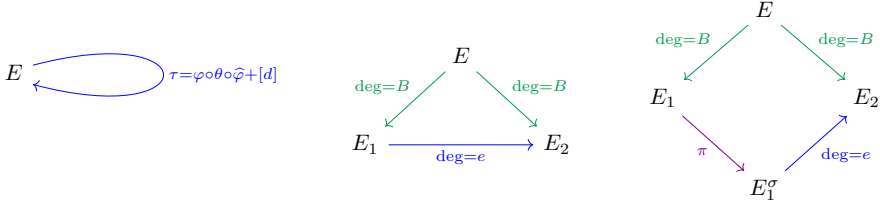
\begin{figure}[h]
\begin{adjustbox}{max totalsize={.8\textwidth}{.24\textheight},center}
\begin{tikzcd}[row sep=2.5em]
E \arrow[loop right,  blue,distance=8em, start anchor={[yshift=1ex]east}, end anchor={[yshift=-1ex]east}]{}{\tau = \phi \circ \theta \circ \widehat{\phi} + \bd} &&\\
\end{tikzcd}
\hspace{7 mm}
\begin{tikzcd}[row sep=2.5em]
&E \arrow[dl,ForestGreen,"\deg=B",swap] \arrow[dr,ForestGreen,"\deg=B"]  &\\ 
E_{1}  \arrow[rr,blue,"\deg=e",swap] && E_{2}   
\end{tikzcd}
\hspace{7 mm}
\begin{tikzcd}[row sep=2.5em]
&E \arrow[dl,ForestGreen,"\deg=B",swap] \arrow[dr,ForestGreen,"\deg=B"]  &\\ 
E_{1}  \arrow[dr,violet,"\pi",swap] && E_{2}  \arrow[dl,blue,<-,"\deg=e"]  \\
&E_{1}^\sigma&
\end{tikzcd}
\end{adjustbox}
\caption{A decomposition of $\tau$ in our two new attacks. Note: we take the dual of one isogeny
in the middle and right-hand diagrams to reverse its arrow.}\label{FigDiamond}
\end{figure}

In both attacks we find $\tau$ by calculating each isogeny in the
decomposition of $\tau$.
In particular, we will use
the restriction of $\phi$ to $E_0[\Ntwo]$ given by Alice's public key
to construct the two $B$-isogenies in the decomposition 
(the green arrows in Figure \ref{FigDiamond}).
Again this can be done efficiently if $\theta$ is in a representation 
that can be efficiently evaluated on $E_0[\Ntwo]$.
As $e$ is smooth we can calculate the $e$-isogeny
in the decomposition (the blue arrow)
via brute-force in time $\Ostar(e^\frac{1}{2})$.
As we can always construct the Frobenius map $\pi$ (the purple arrow),
this gives us $\tau$.
The rest of the proof proceeds as with Petit's original attack
assuming $A$  has $O(\log\log p)$ distinct prime factors, 
see Lemma~\ref{lem:shift_lolli} for details.

\begin{remark}
These methods are an improvement over Petit's original attack, which 
only utilized a shifted lollipop endomorphism $\tau$ of degree $Be$.
There $\tau$ could only be decomposed into
two isogenies of degree $B$ and  $e$ as in Figure \ref{FigOrigPetit}.
Intuitively, Petit's original attack was less effective as
a smaller proportion of $\tau$ could
be calculated directly, and hence a much larger (potentially exponential) proportion
of the endomorphism needed to be brute forced.
It is not clear how to find a better decomposition with more computable isogenies
than those given in Figure~\ref{FigDiamond}
using the fixed parameters and public keys given in SIDH protocols.
Furthermore, 
we give reductions both to \ref{prob:lolli}$_{B^2}$ and \ref{prob:lolli}$_{B^2p}$,
as increasing the degree of $\tau$ does not necessarily make a
shifted lollipop endomorphism $\tau$ easier to find.
\end{remark}

Once an appropriate $(\theta,d,e)$ is found for a particular setting (that is, a
particular choice of $p,A,B,E_0$), then the reduction outlines an algorithm 
that can be 
run to find any unknown isogeny $\phi:E_0\rightarrow E$.
In other words, there is first a precomputation needed to solve \ref{prob:lolli}$_N$ and find a particular  $(\theta,d,e)$.
Using this $(\theta,d,e)$, the above reduction
gives a key-dependent algorithm to find a particular unknown isogeny $\phi:E_0\rightarrow E$.

We now outline how to solve \ref{prob:lolli}$_N$ when $N=B^2p$
for a particular choice of $E_0$,
see Algorithm~\ref{alg:frob-eqn} for details.
A similar technique works when $N=B^2$,
see Algorithm~\ref{alg:death-eqn}.
In most supersingular isogeny-based protocols,
the endomorphism ring of $E_0$ is known.  
A common choice of starting curve, in SIKE for example\footnote{%
    Note that the newest version of SIKE \cite{azarderakhsh2019supersingular} changed the starting curve to a 2-isogenous neighbour of $j=1728$, but this does not affect the asymptotic complexity of any attack.
    }, is where $E_0$
has \jinvariant 1728.
%whose endomorphism ring is (up to small denominators) generated by the Frobenius $\pi:(x,y)\mapsto(x^p,y^p)$ and the automorphism $\iota: (x,y)\mapsto (-x,\sqrt{-1}\cdot y)$.
We show that in the Frobenius isogeny attack finding a shifted lollipop endomorphism of degree $B^2pe$ reduces to finding a solution of
\begin{equation}\label{Petit_equation}
\None^2(a^2+b^2)+pc^2= \Ntwo^2e.
\end{equation}

To proceed choose $c$ and $e$ such that $pc^2=\Ntwo^2e$ modulo $\None^2$.
The remaining equation $a^2+b^2=\frac{\Ntwo e-pc^2}{\None^2}$ can be solved by Cornacchia's algorithm a large percentage of time; else the procedure is restarted with a new choice of $e$ or $c$.

This method of solving \ref{prob:lolli}$_N$ can be used to attack the $n$-party group key agreement \cite{DBLP:journals/iacr/AzarderakhshJJS19}.
We analyze this attack in Section~\ref{ss:frob_nonpoly},
and show that it can be expected, heuristically, to run in polynomial time 
for $n\geq 6$.
The results are summarized in Table~\ref{table:group-key-exchange}, 
and an implementation of this attack for $n=6$ 
can be found at~\url{https://github.com/torsion-attacks-SIDH/6party}.

While we use the Frobenius isogeny attack to highlight vulnerabilities in the isogeny-based group key agreement,
we use the ideas from the dual isogeny attack to investigate situations, namely  different starting curves and base fields, which would result in insecure schemes.

%% file: attacks.tex
\section{Improved torsion-point attacks}
\label{sec:improvements}

In this section, we generalize and improve upon the torsion-point attacks from Petit's 2017 paper~\cite{Petit2017};
in our notation, Petit's attack can be viewed as a reduction of \ref{torsion} to \ref{prob:lolli}$_{B,\lambda}$
together with $\Ostar(1)$-time algorithm to solve \ref{prob:lolli}$_{B,\lambda}$ for certain parameter sets.
In Subsection~\ref{sec:improved-torsion}, we introduce
two new reductions from~\ref{torsion} to \ref{prob:lolli}$_{N,\lambda}$, where $N=B^2$ and $N=B^2p$, respectively.
The runtime of both reductions is
$\Ostar(\lambda^{\frac{1}{2}})$.
The reductions exploit two new techniques:
a dual isogeny and the Frobenius isogeny.

In Subsection \ref{sec:solving-norm} we 
give an algorithm to solve \ref{prob:lolli}$_N$ for $N=B^2$ and $N=B^2p$,
for specific starting curve\footnote{%
    More generally, these attacks apply for any \enquote{special} starting curve in the sense of~\cite{Kohel2014}.
}
$E_0$ under explicit, plausible heuristics (Heuristic~\ref{heur:deathattack-kate} and~\ref{heur:deathattack-kate2}, respectively).
For certain parameters 
these algorithms solve \ref{prob:lolli}$_{N,\lambda}$ for $N=B^2$ for 
$\lambda = \Ostar(1)$ in polynomial time
and \ref{prob:lolli}$_N$ for $N=B^2p$ for $\lambda =\bigO(\log p)$ in polynomial time.
For these parameters, this solves \ref{torsion}
in time $\Ostar(1).$

\subsection{Improved torsion-point attacks}
\label{sec:improved-torsion}

The main ingredient in Petit's~\cite{Petit2017} attack can be viewed  as a reduction of \ref{torsion} to \ref{prob:lolli}$_B$. 
In this section we introduce our first extension of this attack: the \emph{dual isogeny attack},
which works by exploiting the dual isogeny of 
the (shifted lollipop) endomorphism $\tau$ on $E$. 
We begin by giving the reduction for the dual isogeny attack.

\begin{theorem}\label{more efficient}
Suppose we are given an instance of~\ref{torsion} where $A$ has $O(\log\log p)$ distinct prime factors.
Assume we are given
the restriction of  a trace-zero endomorphism $\theta\in \End(E_0)$
to $E_0[B]$,
an integer $d$ coprime to $B$, 
and a smooth integer $e$ such that
$$\deg(\phi\circ\theta\circ\hat{\phi}+[d])=B^2e. $$
Then we can compute $\phi$ in time 
    $\Ostar(\hspace{-.2em}\sqrt{e})=\bigO(\hspace{-.2em}\sqrt{e}\cdot\polylog(p))$.
\end{theorem}
We first state a technical lemma which mostly follows from~\cite[Section 4.3]{Petit2017}.

\begin{lemma}\label{lem:shift_lolli}
Let $\None$ be a smooth integer with $O(\log\log p)$ distinct prime factors, and let $E_0/\F_{p^2}$ and $E/\F_{p^2}$ be two supersingular elliptic curves
connected by an unknown degree-$\None$ isogeny $\phi$.
Suppose we are given the restriction of some $\tau\in \End(E)$ to $E[A]$,
where $\tau$ is of the form $\tau=\phi\circ\theta\circ\hat{\phi}+[d]$ 
such that if $E[m] \subseteq \ker \tau$ then $m\mid 2$.
Then we can compute $\ker \phi$ in time $\Ostar(1)$.
\end{lemma}

\CP{Compared to ac17, this restricts $m$ and $\tau$, I suppose to guarantee the complexity. Would it make sense (for teh sake of simplicity/readability) to further restrict to $m=1$ (and maybe remark about general version in appendix)? It may make the statement easier to read, and I am right this won't affect the results significantly? (I know you have been discussing this extensively and I did not follow the discussion, so I will skip the proof and let you decide on what's best)}

\begin{proof}
    \iffinal
    See the full version~\cite[Appendix A.1]{fullversion}.
    \else
    See Appendix~\ref{proof-shift_lolli}.
    \fi
\end{proof}

\begin{proof}[of Theorem~\ref{more efficient}]
Suppose we have $d,e$ and the restriction of $\theta$ to $E[B]$ satisfying the conditions above.
We wish to find an explicit description of $\tau=\phi\circ\theta\circ\hat{\phi}+[d]$.  
Let $m$ be the largest integer dividing $B$ such that $E[m] \subseteq \ker \tau$.
Since the degree of $\tau$ is $B^2e$, there exists a decomposition of the form
$\tau= \psi' \circ \eta \circ \psi \circ [m]$, where $\psi$ and $\psi'$ are isogenies of degree $B/m$, $\psi$ is cyclic,
and $\eta$ is an isogeny of degree $e$. 

We proceed by deriving the maps in this decomposition. 
Since $\tau$ factors through $[m]$, this implies $m$ divides $\tr(\tau)=2d$.
As we chose $d$ coprime to $B$, this shows $m \in \{1,2\}$.

To compute $\psi$ and $\psi'$, we start by finding the restriction of $\tau$ to the $B$-torsion.
This can be computed from what we are given:  the restrictions of $\theta$, $[d]$, $\phi$, hence $\hat{\phi}$, to the $B$-torsion of the relevant elliptic curves. 
This also allows us to compute $m$ explicitly, 
as the largest integer dividing $B$ such that $E[m] \subseteq \ker\tau \cap E[B]$.

Let $\tau' = \psi' \circ \eta \circ \psi$.
The isogeny $\psi$ can now be computed from the restriction of $\tau$ to $E[B]$ via
$$\ker \psi =  \ker \tau' \cap ([m] \cdot E[B]) = (\ker\tau \cap E[B])/E[m].$$  
From the cyclicity of $\psi$, we can also deduce that $\ker\hat{\psi'} = \tau(E[B])$, which gives $\psi'$ explicitly.

Finally,
%for each choice of $\psi'$,
we recover the isogeny $\eta$ by a generic meet-in-the-middle algorithm,
    which runs in time $\Ostar(\hspace{-.2em}\sqrt{e})$ since $e$ is smooth.
    Note that if $e = \Ostar(1)$, then the entire algorithm runs in time $\polylog(p)$.
In this way we have found $\tau$ explicitly, and by Lemma~\ref{lem:shift_lolli} can compute $\phi$.
\qed
\end{proof}

Next we give the reduction for the \textit{Frobenius isogeny attack},
which works by exploiting the Frobenius isogeny on $E$
to improve, or at least alter, the dual attack.

\begin{theorem}\label{frob-thm-kate}
Suppose we are given an instance of~\ref{torsion} where $A$ has at most $\bigO(\log\log p)$ distinct prime factors.
Assume we are given the restriction of a trace-zero endomorphism $\theta \in \End(E_0)$ to $E_0[B]$, an integer $d$ coprime to $B$, and a smooth integer $e$ such that
\[
\deg( \phi \circ \theta \circ \hat{\phi} + [d] ) = B^2 p e
\,\text.
\]
Then we can compute $\phi$ in time 
    $\Ostar(\hspace{-.2em}\sqrt{e})=\bigO(\hspace{-.2em}\sqrt{e}\cdot\polylog(p))$.
\end{theorem}

\begin{proof}
Let $\tau = \phi \circ \theta \circ \hat{\phi} + [d]$.  As in the proof of Theorem \ref{more efficient}, we can decompose $\tau$ as $\psi' \circ \eta \circ \psi \circ [m]$, where $\eta$ has degree $pe$, and compute $\psi$ and $\psi'$ efficiently.

We are left to recovering $\eta$.  Instead of using a generic meet-in-the-middle algorithm, we observe that $\eta$ has inseparable degree $p$ (since we are in the supersingular case).  Thus, $\eta = \eta' \circ \pi$, where $\pi$ is the $p$-power Frobenius isogeny, and $\eta'$ is of degree $e$. % (Silverman, II.2.12).  
We use the meet-in-the-middle algorithm 
on $\eta'$ and recover the specified runtime.
\qed
\end{proof}

\begin{remark}
It is a natural question why we stick to the $p$-power Frobenius and why the attack doesn't give a better condition for a higher-power Frobenius isogeny. The reason is that for supersingular elliptic curves defined over $\mathbb{F}_{p^2}$, the $p^2$-power Frobenius isogeny is just a scalar multiplication followed by an isomorphism (since every supersingular $j$-invariant lies in $\mathbb{F}_{p^2}$), and hence would already be covered by the method of Theorem~\ref{more efficient}.

More generally, see Section~\ref{sec:open} for a more abstract viewpoint that subsumes both of the reductions given above (but has not led to the discovery of other useful variants thus far).
\end{remark}

The complexity of both attacks relies on whether one can find a suitable endomorphism $\theta$ with $e$ as small as possible. In the next subsection we 
will establish criteria when we can find a suitable $\theta$ when the starting curve has $j$-invariant 1728.

\subsection{Solving norm equations}
\label{sec:solving-norm}

In Subsection~\ref{sec:improved-torsion}
we showed two reductions (Theorem~\ref{more efficient} and Theorem~\ref{frob-thm-kate})
from~\ref{torsion}
to \ref{prob:lolli}$_N$ where $N=B^2$ and $N=B^2p.$ 
To complete the description of our attacks, we discuss how
to solve~\ref{prob:lolli}$_N$ in these two cases;
that is, we want to find solutions ($\theta,d,e$) to 
$$\deg(\phi\circ\theta\circ\hat{\phi}+[d])=\None^2 \deg\theta + d^2=Ne,$$
where $N=B^2$ or $N=B^2p$.

The degree of any endomorphism of $E_0$ is represented by
a quadratic form that depends on $E_0$.
To simply our exposition we choose
$E_0/\F_p\colon y^2 = x^3 + x$ (having $j=1728$), where $p$ is congruent to $3\pmod 4$.
In this case the endomorphism ring $\End(E_0)$ has a particularly simple norm form.
To complete the dual isogeny attack, it suffices to find a solution to the norm
Equation (\ref{the_eqnofdeath}):

\begin{corollary}\label{eqnofdeath}
Let $p\equiv3\pmod4$
and $j(E_0)=1728$. 
Consider coprime smooth integers $A,B$ such that $A$ has (at most) $O(\log\log p)$ distinct prime factors and suppose that we are given an integer solution
$(a,b,c,d,e)$,
with $e$ smooth,
to the equation
  \begin{equation}
        \label{the_eqnofdeath}
        A^2(p a^2 + p b^2 + c^2 )+ d^2 =  B^2 e
        \,\text.
    \end{equation}
Then we can solve~\ref{torsion} with the above parameters in time $\Ostar(\hspace{-.2em}\sqrt{e})$.
\end{corollary}
\begin{proof}
    Let $\iota \in \End(E_0)$ be such that $\iota^2=[-1]$ and let $\pi$ be the Frobenius endomorphism of $E_0$. 
    Let $\phi$ be as in Theorem~\ref{more efficient}.
    The endomorphism $\theta = a\iota\pi + b\pi + c\iota$ and the given choice of $d$ satisfies the requirements of Theorem~\ref{more efficient}.
\qed
\end{proof}
To complete the Frobenius isogeny attack, we find a solution to the norm
equation (\ref{eqnoflessdeath}):

\begin{corollary}\label{eqnoflessdeath}
Let $p\equiv3\pmod4$
and $j(E_0)=1728$. 
Consider coprime smooth integers $A,B$ such that $A$ has (at most) $O(\log\log p)$ distinct prime factors and suppose that we are given an integer solution
$(a,b,d,e)$,
with $e$ smooth,
to the equation
  \begin{equation}
        \label{the_eqnoflessdeath}
        A^2(a^2 + b^2 )+ pd^2 =  B^2 e
        \,\text.
    \end{equation}
Then we can solve~\ref{torsion} with the above parameters in time $\Ostar(\hspace{-.2em}\sqrt{e})$.
\end{corollary}
\begin{proof}
    With $\iota$ and $\pi$ as in the proof of Corollary \ref{eqnofdeath}, and $\phi$ as in Theorem~\ref{frob-thm-kate},
    the endomorphism $\theta = a\iota\pi + b\pi$, together with the choice $c=0$ satisfies the requirements of Theorem~\ref{frob-thm-kate} (to see this, multiply \eqref{the_eqnoflessdeath} through by $p$).
\qed
\end{proof}

Now we present two algorithms for solving each norm equation (\ref{the_eqnofdeath})
and (\ref{the_eqnoflessdeath}). The algorithms are similar in nature but they work on different parameter sets. See Algorithms \ref{alg:death-eqn} and \ref{alg:frob-eqn}.

\begin{algorithm}
    \caption{\,Solving norm equation~\ref{the_eqnofdeath}.} \label{alg:death-eqn}
    \vspace{.2ex}
    \Input {%
            SIDH parameters $p,A,B$.
        }
    \Output {%}
        A solution $(a,b,c,d,e)$
        to \eqref{the_eqnofdeath}.
        }
    %
    % precomputation
    %
    \vspace{.4ex}
           Set $e:=2$.
           \;
           \If{$e$ \textnormal{is a quadratic non-residue mod} $A^2$\label{step:compute-d}}{
           Set $e := e+1$ and go to Step~\ref{step:compute-d}.\label{step:qre}}
           Compute $d$ such that $d^2\equiv eB^2\pmod{A^2}$.
           \;
    \If{$eB^2 - d^2$ \textnormal{ is a quadratic non-residue mod} $p$\label{step:qr-mod-p}}{
        Set $e:=e+1$ and go to Step~\ref{step:compute-d}.\label{step:qrd}
        \;
        }
        Compute $c$ as the smallest positive integer such that
        $c^2A^2 \equiv eB^2-d^2 \pmod p$.\label{computec}
        \;
    \If{$eB^2>d^2+c^2A^2$}{
        \If{$\frac{eB^2-d^2-c^2A^2}{A^2p}$ \textnormal{ is prime}\label{step:isprime}}{
            \If{$\frac{eB^2-d^2-c^2A^2}{A^2p} \equiv 1 \pmod {4}$\label{step:is1mod4}}{
                   Find $a,b \in \Z$ such that $a^2+b^2=\frac{eB^2-d^2-c^2A^2}{A^2p}$.
                   \;
                   \Return{ $(a,b,c,d,e)$.}
                }
            }
            Set $e:= e+1$ and go to Step~\ref{step:compute-d}.
        }
    \Else{
     \Return{\textnormal{Failure.}}
     }
\end{algorithm}

\begin{algorithm}
    \caption{\,Solving norm equation~\ref{the_eqnoflessdeath}.} \label{alg:frob-eqn}
    \vspace{.2ex}
    \Input {%
            SIDH parameters $p,A,B$.
        }
    \Output {%
        A solution $(a,b,d,e)$
        to \eqref{the_eqnoflessdeath}.
        }
    %
    % precomputation
    %
    \vspace{.4ex}
           Set $e:=1$.
           \;
    \While{$\frac{eB^2}{p}$ \textnormal{ is a quadratic non-residue mod } $A^2$
        \label{step:inc-e}}
        {
        Set $e:=e+1$.
        \;
        }
        Compute $d$ such that
        $eB^2 \equiv pd^2 \pmod {A^2}$.
        \; \label{step:findc-frob}
    \If{$eB^2>pd^2$}{
        \If{$\frac{eB^2-pd^2}{A^2}$ \textnormal{ is prime}\label{step:isprime2}}{
            \If{$\frac{eB^2-pd^2}{A^2} \equiv 1 \pmod {4}$\label{step:is1mod42}}{
                   Find $a,b \in \Z$ such that $a^2+b^2=\frac{eB^2-pd^2}{A^2}$.
                   \;\label{step:cornacchia-frob}
                   \Return{ $(a,b,d,e)$.}
                }
            }
            Set $e:= e+1$ and go to Step~\ref{step:inc-e}.
        }
    \Else{
     \Return{\textnormal{Failure.}}
     }
\end{algorithm}

\subsection{Runtime and justification for Algorithms \ref{alg:death-eqn} and \ref{alg:frob-eqn}}

The remainder of this section is devoted to providing justification that the algorithms succeed in polynomial time.

\begin{heuristic}\label{heur:deathattack-kate}
Let $p, A, B$ be SIDH parameters.
Note that for each $e$, the equation
\begin{equation}
\label{eqn:heur1-kate}
eB^2 = d^2 + c^2A^2 \pmod{A^2 p},
\end{equation}
may or may not have a solution $(c,d)$.  We assert two heuristics:
\begin{enumerate}
    \item \label{heurdeath-kate-pt1} Amongst invertible residues $e$ modulo $A^2p$, which are quadratic residues modulo $A^2$, the probability of the existence of a solution is approximately $1/2$.
    \item \label{heurdeath-kate-pt2} Amongst those $e$ for which there is a solution, and for which the resulting integer
\begin{equation}\label{eq:heurdeath-kate}
\frac{B^2e-d^2-c^2A^2}{A^2p}
\end{equation}
is positive, the probability that \eqref{eq:heurdeath-kate} is a prime congruent to $1$ modulo $4$ is expected to be approximately the same as the probability that a random integer of the same size is prime congruent to $1$ modulo $4$.
\end{enumerate}
\end{heuristic}

\textit{Justification.}
By the Chinese remainder theorem, solving \eqref{eqn:heur1-kate} amounts to solving $eB^2 \equiv d^2 \pmod{A^2}$ and $eB^2 \equiv d^2 + c^2A^2 \pmod{p}$.  If $e$ is a quadratic residue modulo $A^2$, then the first of these equations has a solution $d$.  Using this $d$, the second equation has either no solutions or two, with equal probability.  This justifies the first item.

For the second item, this is a restriction of the assertion that the values of the quadratic function $B^2e - d^2 - c^2A^2$, in terms of variables $e$, $c$ and $d$, behave, in terms of their factorizations, as if they were random integers.  In particular, the conditional probability that the value has the form $A^2pq$ for a prime $q \equiv 1 \pmod 4$, given that it is divisible by $A^2p$, is as for random integers.

\begin{proposition}\label{prop:eq1} Let $\epsilon > 0$.
Under Heuristic~\ref{heur:deathattack-kate}, 
if $B>pA$ and $p>A$, but $B$ is at most polynomial in $A$, then Algorithm \ref{alg:death-eqn} returns a solution $(a,b,c,d,e)$ with $e = O(\log^{2+\epsilon}(p))$ in polynomial time.
\end{proposition}
\begin{proof}
Checking that a number is a quadratic residue modulo $p$ can be accomplished by a square-and-multiply algorithm.
Checking that a certain number is prime can also be accomplished in polynomial-time.
Representing a prime as a sum of two squares can be carried out by Cornacchia's algorithm.  Suppose one iterates $e$ a total of $X$ times.

For the algorithm to succeed, we must succeed in three key steps in reasonable time:  first, that $e$ such that $e$ is a quadratic residue modulo $A^2$ (Step \ref{step:compute-d}) and second, that $eB^2-d^2$ is a quaratic residue modulo $p$ (Step \ref{step:qr-mod-p}), and third, that $\frac{eB^2-d^2-c^2A^2}{A^2p}$ is a prime congruent to $1$ modulo $4$ (Step \ref{step:isprime}--\ref{step:is1mod4}).  Suppose we check values of $e$ up to size $X$.

For Step \ref{step:compute-d}, it suffices to find $e$ an integer square, which happens $1/\sqrt{X}$ of the time.
When this is satisfied, the resulting $d$ can be taken so $d < A^2$.  For Step \ref{step:qr-mod-p}, under Heuristic~\ref{heur:deathattack-kate} Part~\ref{heurdeath-kate-pt1}, the probability that a corresponding $c$ exists is $1/2$.  Such a $c$ can be taken with $c < p$.  Under the given assumption that $B>pA$ and $p>A$, then
\[eB^2 \geq 2B^2 > 2p^2A^2 > p^2A^2 + A^4 > c^2A^2 + d^2.\]
So the quantity in Heuristic~\ref{heur:deathattack-kate} Part~\ref{heurdeath-kate-pt2} is positive.  We can bound it by $eB^2/pA^2$.  Since $B$ is at worst polynomial in $A$, the quantity $B^2/pA^2$ is at worst polynomial in $p$, say $p^k$.  Hence, for Step \ref{step:isprime}--\ref{step:is1mod4}, one expects at a proportion $1/\log(p^kX)$ of successes to find a prime congruent to $1$ modulo $4$.  Such a prime is a sum of two squares, and the algorithm succeeds.

Finally, we set $X = \log^{2 + \epsilon}(p)$ to optimize the result.  If one iterates $e$ at most $\log^{2+\epsilon}(p)$ times, one expects to succeed at Step \ref{step:compute-d} at least $\log^{1+\epsilon}(p))$ times, to succeed at Step \ref{step:qr-mod-p} half of those times, and to succeed at Steps \ref{step:isprime} and \ref{step:is1mod4} at least $1/\log(p^k\log^{2+\epsilon}(p))$  of those times.  This gives a total probability of success, at any one iteration, of $1/4k\log^{2+\epsilon}(p)$.  Hence we expect to succeed with polynomial probability.
\end{proof}

For the analysis of Algorithm 2, the following technical lemma is helpful.

\begin{lemma}\label{technical-lemma}
Let $M$ be an integer.  Let $r$ be an invertible residue modulo $M$. Then the pattern of $e$ such that $re$ is a quadratic residue repeats modulo $N = 4\operatorname{sqfr}(M)$, four times the squarefree part of $M$.  Among residues modulo $4\operatorname{sqfr}(M)$, a proportion of $1/2^\ell$ of them are solutions, where $\ell$ is the number of distinct primes dividing $M$.
\end{lemma}
\begin{proof}
Suppose $M$ has prime factorization $M = \prod_{i} l_i^{e_i}$.  A residue $x$ modulo $M$ is a quadratic residue if and only if it is a quadratic residue modulo $l_i^{e_i}$ for every $i$. For odd $l_i$, a residue modulo $l_i^{e_i}$ is a quadratic residue if and only if it is a quadratic residue modulo $l_i$, by Hensel's lemma.  And a residue modulo $2^{e}$, $e \ge 3$, is a quadratic residue if and only if it is a quadratic residue modulo $8$.  By the Chinese remainder theorem, $re$ is a quadratic residue modulo $M$ if and only if $re$ is a quadratic residue modulo $4\operatorname{sqfr}(M)$.
\end{proof}

\begin{heuristic}\label{heur:deathattack-kate2}
Let $p, A, B$ be SIDH parameters.  Let $\ell$ be the number of distinct prime divisors of $A$.
Note that for each $e$, the equation
\begin{equation}
\label{eqn:heur2-kate}
eB^2 = pd^2 \pmod{A^2}
\end{equation}
may or may not have solutions $d$.  We assert two heuristics:
\begin{enumerate}
    \item \label{heurdeath-kate2-pt1} As $e$ varies, the probability that it has solutions is $1/2^\ell$.
    \item \label{heurdeath-kate2-pt2} Amongst those $e$ for which there is a solution, and for which the resulting integer
\begin{equation}\label{eq:heurdeath-kate2}
\frac{B^2e-pd^2}{A^2}
\end{equation}
is positive, the probability that \eqref{eq:heurdeath-kate2} is a prime congruent to $1$ modulo $4$ is expected to be approximately the same as the probability that a random integer of the same size is prime congruent to $1$ modulo $4$.
\end{enumerate}
\end{heuristic}

\textit{Justification.}
Consider the first item.  Modulo each prime dividing $A^2$, the quadratic residues vs. non-residues are expected to be distributed \enquote{randomly}, resulting in a random distribution modulo $4\operatorname{sqfr}(A)$, by Lemma \ref{technical-lemma}.  %Note that under GRH, the least non-residue has been shown to lie below $O(\log^2(p))$ \KS{cite MR3356031}.

For the second item, this is a restriction of the assertion that the values of the quadratic function $B^2e - pd^2$, in terms of variables $e$ and $d$, behave, in terms of their factorizations, as if they were random integers.  In particular, the conditional probability that the value has the form $A^2q$ for a prime $q \equiv 1 \pmod 4$, given that it is divisible by $A^2$, is as for random integers.  %The study of the values of quadratic functions relates to the class number of the associated field, in this case $\mathbb{Q}(\sqrt{pe})$, and is more developed than the study of values of translated quadratic forms, as in the previous heuristic.  However, we know of no reason this is not a reasonable heuristic, at least up to a constant factor. \KS{should say something better here?}

\begin{proposition}\label{prop:eq2}
Under Heuristic~\ref{heur:deathattack-kate2}, if $B > \sqrt{p}A^2$, $A$ has $O(\log\log p)$ distinct prime factors, $B$ is at most polynomial in $A$, and\footnote{In the proof, it suffices to take $p^k > A$ for any $k$.} $p > A$, 
then Algorithm~\ref{alg:frob-eqn} returns a solution $(a,b,d,e)$ with $e = O(\log p)$ in polynomial time.
\end{proposition}

\begin{proof}
Checking that a number is a quadratic residue can be accomplished by a square-and-multiply algorithm. 
Checking that a certain number is prime can also be accomplished in polynomial-time.
Representing a prime as a sum of two squares can be carried out by Cornacchia's algorithm. 

For the algorithm to succeed, we must succeed in two key steps in reasonable time:  first, that $e$ such that $eB^2/p$ is a quadratic residue (Step \ref{step:inc-e}) and second, that $\frac{eB^2-pd^2}{A^2}$ is a prime congruent to $1$ modulo $4$ (Step \ref{step:isprime2}--\ref{step:is1mod42}).  Suppose we check values of $e$ up to size $X$.

By Heuristic~\ref{heur:deathattack-kate2} Part~\ref{heurdeath-kate-pt1}, we expect to succeed at Step \ref{step:inc-e} with probability $1/2^\ell$, where $\ell$ is the number of distinct prime divisors of $A$.  \KS{Needed here?}

When this is satisfied, the resulting $d$ can be taken so $d < A^2$.  
Under the given assumption that $B>\sqrt{p}A$, then
\[eB^2 \geq B^2 > pA^4 > pd^2.\]
So the quantity in Heuristic~\ref{heur:deathattack-kate2} Part~\ref{heurdeath-kate-pt2} is positive.  We can bound it above by $eB^2/A^2$, and using the assumption that $B$ is at most polynomial in $A$, we bound this by $<p^kX$ for some $k$.  So we expect to succeed in Step \ref{step:isprime2}--\ref{step:is1mod42} with probability $1/2\log(p^kX)$.  The resulting prime is a sum of two squares, and the algorithm succeeds.
Thus, taking $X = O(\log p)$ suffices for the statement.
\end{proof}

\begin{remark}
In practice, in Algorithm~\ref{alg:frob-eqn} it may be more efficient to increment $d$ by multiples of $A^2$ 
in place of incrementing $e$.
This however makes the inequalities satisfied by $A$, $B$, and $p$ slightly less tight so for the sake of cleaner results we opted for incrementing only $e$.
\end{remark}

\begin{remark}
If parameters $A$ and $B$ are slightly more unbalanced (i.e., $B>rA^2\sqrt{p}$ for some $r>100$), then instead of increasing $e$ it is better to fix $e$ and increase $d$ by $A^2$ in each step. 
\end{remark}

%% file: weak.tex
\section{Backdoor instances} \label{sec:insecure}

In this section we give a method to specifically create instantiations of the SIDH framework for which we can solve~\ref{torsion} more efficiently.
So far all of our results were only considering cases where the starting curve $E_0$ has \jinvariant 1728. In Section~\ref{subsec:backdoor} we explore the question: For given $A,B$ can we construct starting curves for which we can solve~\ref{torsion} with a better balance? We will call such curves \emph{backdoor curves} (see Definition~\ref{insecure}), and quantify the number of backdoor curves in Section \ref{sec:numweak}.
In Sections~\ref{ss:Insecure p} and~\ref{subsec:param}, we also consider backdoored choices of $(p,A,B)$, for which we can solve \ref{torsion} more efficiently even when starting from the curve with \jinvariant~1728.

\subsection{Backdoor curves}\label{subsec:backdoor}

This section introduces the concept of \emph{backdoor curves} and how to find such curves. 
Roughly speaking,
these are specially crafted curves which, 
if used as starting curves for the SIDH protocol, 
are susceptible to our dual isogeny attack
by the party which chose the curve,
under only moderately unbalanced parameters $A,B$;
in particular, the imbalance is independent of $p$.
In fact, when we allow for non-polynomial time attacks
we get an asymptotic improvement over meet-in-the-middle
for balanced SIDH parameters
(but starting from a backdoor curve).
These curves could potentially be utilized as a backdoor,
for example by suggesting the use of such a curve as a standardized starting curve.
We note that
it does not seem obvious how backdoored curves,
such as those generated by Algorithm~\ref{alg:insecure},
can be detected by other parties:
The existence of an endomorphism of large degree which satisfies
Equation~\ref{the_eqnofdeath}
does not seem to be detectable without trying to recover such an endomorphism,
which is hard using all currently known algorithms.
The notion of backdoor curves is dependent on the parameters $A,B$, which motivates the following definition:

\begin{definition}\label{insecure}
Let $A,B$ be coprime positive integers. 
An \emph{$(A,B)$-backdoor curve} is a tuple $(E_0,\theta,d,e)$, where
$E_0$ is a supersingular elliptic curve  defined over some $\F_{p^2}$,
an endomorphism $\theta \in \End(E_0)$ in an efficient representation,
and two integers $d,e$ such that Algorithm~\ref{alg:find_isog} solves~\ref{torsion}
for that particular $E_0$ in time polynomial in $\log p$ when given $(\theta,d,e)$.
\end{definition}

The main result of this section is Algorithm \ref{alg:insecure} which computes $(A,B)$-backdoor curves in heuristic polynomial time,
assuming we have a factoring oracle (see Theorem~\ref{trapdoor}).

\begin{algorithm}%[!htp]
    \caption{\,Generating $(A,B)$-backdoor curves.} \label{alg:insecure}
    \vspace{.2ex}
    \Input{
        A prime $p\equiv3\pmod4$ and
        smooth coprime integers $A,B$ with $B>A^2$.
    }
    \Output{
        An $(A,B)$-backdoor curve $(E_0,\theta,d,e)$
        with $E_0/\F_{p^2}$.
    }
    \vspace{.4ex}
	Set $e:=1$. \label{insecure_e}
    \;
    \While{$\mathtt{true}$} {
        Find an integer $d$ such that $d^2\equiv B^2e \pmod {A^2}$\label{d}.
        \;
        \If {$d$ is coprime to $B$} {
            \If {$\frac{B^2e-d^2}{A^2}$ is square modulo $p$} {
                Find rational $a,b,c$ such that $pa^2+pb^2+c^2=\frac{B^2e-d^2}{A^2}$.\label{step:solve}
                \;
                \Break
            }
        }
        Set $e$ to the next square.
    }
    \vspace{.2ex}
    Set $\vartheta=a\ij+b\j+c\i\in B_{p,\infty}$.
    \;
	Compute a maximal order $\OO\subseteq \BB_{p,\infty}$ containing $\theta$.\label{step:maxorder}
    \;
	Compute an elliptic curve $E_0$ whose endomorphism ring is isomorphic to $\OO$.\label{step:elliptic}
    \;
    Construct an efficient representation of the endomorphism $\theta$ of $E_0$ corresponding to~$\vartheta$.\label{step:evalendo}
    \;
    \Return {$(E_0,\theta,d,e)$}.
\end{algorithm}

\begin{theorem}\label{trapdoor}
Given an oracle for factoring, if $A$ has (at most) $O(\log\log p)$ distinct prime factors, then
Algorithm~\ref{alg:insecure} can heuristically be expected to succeed in polynomial time.
\end{theorem}

\begin{remark}
The imbalance $B>A^2$ is naturally satisfied for a group key agreement in the style of \cite{DBLP:journals/iacr/AzarderakhshJJS19} with three or more participants; we can break (in polynomial time) such a variant when starting at an $(A,B)$-backdoor curve.
\end{remark}

\noindent
Before proving Theorem~\ref{trapdoor} we need the following easy lemma: 

\begin{lemma}
Let $p$ be a prime congruent to 3 modulo 4. Let $D$ be a positive integer. Then the quadratic form $Q(x_1,x_2,x_3,x_4)=px_1^2+px_2^2+x_3^2-Dx_4^2$ has a nontrivial integer root if and only if $D$ is a quadratic residue modulo $p$.
\label{quadratic}
\end{lemma}
\begin{proof}
The proof is essentially a special case of \cite[Proposition 10]{simon2005quadratic}, but we give a brief sketch of the proof here. If $D$ is a quadratic residue modulo $p$, then $px_1^2+px_2^2+x_3^2-Dx_4^2$ has a solution in $\mathbb{Q}_p$ by setting $x_1=x_2=0$ and $x_4=1$ and applying Hensel's lemma to the equation $x_3^2=D$. The quadratic form $Q$ also has local solutions everywhere else (the 2-adic case involves looking at the equation modulo 8 and applying a 2-adic version of Hensel's lemma). If on the other hand $D$ is not a quadratic residue modulo $p$, then one has to choose $x_3$ and $x_4$ to be divisible by $p$. Dividing the equation $Q(x_1,x_2,x_3,x_4)=0$ by $p$
and reducing modulo $p$ yields $x_1^2+x_2^2 \equiv 0\pmod p$.
This does not have a solution as $p\equiv3\pmod4$. Finally, one can show that this implies that $Q$ does not have a root in $\mathbb{Q}_p$.
\qed
\end{proof}

\begin{proof}[of Theorem~\ref{trapdoor}]
The main idea is to apply Theorem~\ref{more efficient} in the following way:
using Algorithm~\ref{alg:insecure},
we find integers $D$, $d$, and $e$, with $e$ polynomially small and $D$ a quadratic residue mod $p$, 
such that $A^2D+d^2=B^2e$,
and an element $\theta \in \BB_{p,\infty}$ of trace zero and such that $\theta^2 = -D$.
We then construct a maximal order $\OO \subseteq \BB_{p,\infty}$ containing $\theta$ 
and an elliptic curve $E_0$ with $\End(E_0) \cong \OO$.

Most steps of Algorithm~\ref{alg:insecure} obviously run in polynomial time, 
although some need further explanation. 
We expect $d^2 \approx A^4$ since we solved for $d$ modulo $B^2$,
and we expect $e$ to be small since heuristically we find a quadratic residue after a small number of tries.
Then the right-hand side in step \ref{step:solve} should be positive since $B > A^2$,
so by Lemma~\ref{quadratic},
step \ref{step:solve} returns a solution using Simon's algorithm~\cite{simon2005quadratic}, assuming an oracle for factoring $\frac{B^2e-d^2}{A^2}$. 
For step \ref{step:maxorder}, we can apply either of the polynomial-time algorithms \cite{ivanyos1993finding, voight2013identifying} for finding maximal orders containing a fixed order in a quaternion algebra,
which again assume a factoring oracle.
    Steps~\ref{step:elliptic} and~\ref{step:evalendo} can be accomplished using the heuristically polynomial-time algorithm from \cite{DBLP:journals/iacr/PetitL17,eisentrager2018supersingular} which returns both the curve~$E_0$ and (see~\cite[\S\,5.3, Algorithm~5]{eisentrager2018supersingular}) an efficient representation of~$\theta$.
\qed
\end{proof}

\begin{remark}
The algorithm uses factorization twice (once in solving the quadratic form and once in factoring the discriminant of the starting order).
\iffinal
In the full version~\cite[Appendix~C]{fullversion}
\else
In Appendix~\ref{imp}
\fi
we discuss how one can ensure in practice that the numbers to be factored have an easy factorization.
\end{remark}
\begin{remark}
Denis Simon's algorithm \cite{simon2005quadratic} is available on his webpage.%
\footnote{\url{https://simond.users.lmno.cnrs.fr/}}
Furthermore, it is implemented in MAGMA \cite{bosma1997magma} and PARI/GP \cite{batut2000user}.
The main contribution of Simon's paper is a polynomial-time algorithm for finding nontrivial zeroes of
(not necessarily diagonal)
quadratic forms
which 
does not rely on an effective version of Dirichlet's theorem. 
In our case, however, we only need a heuristic polynomial-time algorithm for finding a nontrivial zero $(x,y,z,u)$ of a form $px^2+py^2+z^2-Du^2$. 
We sketch an easy way to do this:
Suppose that $D$ is squarefree, and pick a prime $q\equiv1\pmod4$ such that $-pq$ is a quadratic residue modulo every prime divisor of $D$. 
It is then easy to see that the quadratic equations $px^2+py^2=pq$ and $Du^2-z^2=pq$ both admit a nontrivial rational solution which can be found using~\cite{cremona2003efficient}. 
\end{remark}

There are two natural questions that arise when looking at Theorem \ref{trapdoor}: 
\begin{itemize}
    \item Why are we using the dual attack and not the Frobenius attack?
    \item Why do we get a substantially better balance than we had before?
\end{itemize}

The answer to the first question is that we get a better result in terms of balance. In the Frobenius version we essentially get the same bound for backdoor curves as for the curve with \jinvariant 1728. The answer to the second question is that by not restricting ourselves to one starting curve we only have the condition that $pa^2+pb^2+c^2$ is an integer and $a,b,c$ can be rational numbers. 
\begin{remark}
Backdoor curves also have a constructive application: An improvement on 
the recent paper \cite{de2019seta} using Petit's attack to build a one-way function \enquote{S\'{E}TA}. 
In this scheme, the secret key is a secret isogeny to a curve $E_s$ that starts from the elliptic curve with \jinvariant 1728 and the message is the end point of
a secret isogeny from $E_s$ to some curve $E_m$,
together with the image of some torsion points.
The reason for using \jinvariant 1728 is in order to apply 
Petit's attack constructively. 
One could instead use a backdoor curve;
this provides more flexibility to the scheme as one does not need to disclose the starting curve and the corresponding norm equation is easier to solve.
\end{remark}

\subsection{Counting backdoor curves} \label{sec:numweak}

Having shown how to construct backdoor curves and how to exploit them,
a natural question to ask is how many of these curves we can find using the methods of the previous section.
Recall that the methods above search for an element $\vartheta \in \BB_{p,\infty}$ with reduced norm $D$.
Theorem~\ref{thm:onuki} below suggests
they can be expected to produce exponentially (in $\log D$) many different maximal orders,
and using Lemma~\ref{lemma:kaneko} we can prove this rigorously for the (indeed interesting) case of $(A,B)$-backdoor curves with $AB \approx p$ and $A^2 < B < A^3$ (cf.\ Theorem~\ref{trapdoor}).

We first recall some notation from~\cite{onuki2020oriented}.
The set $\rho(\Ell(\OO))$ consists of
the
reductions modulo $p$
of all elliptic curves over $\overline{\Q}$
with complex multiplication by~$\OO$.
Each curve $E=\mathcal E\bmod p$ in this set
comes with an optimal embedding
$\iota\colon\OO\hookrightarrow\End(E)$,
referred to as an \enquote{orientation} of~$E$,
and conversely,
\cite[Prop.~3.3]{onuki2020oriented}~%
shows that\,---\,%
up to conjugation\,---\,%
each oriented curve $(E,\iota)$
defined over~$\overline\F_p$
is obtained by the reduction modulo $p$ of a characteristic-zero curve;
in other words, either $(E,\iota)$ or $(E^{(p)},\iota^{(p)})$
lies in $\rho(\Ell(\OO))$.
The following theorem was to our knowledge first explicitly stated and used
constructively in~\cite{coloKohel} to build the \enquote{OSIDH} cryptosystem.
The proof was omitted,%
\footnote{%
    In \cite{coloKohel}
    the theorem was referred to as a classical result,
    considered to be folklore.
}
but later published by Onuki~\cite{onuki2020oriented},
whose formulation we reproduce here:

\begin{theorem}
    \label{thm:onuki}
    Let $K$ be an imaginary quadratic field such that $p$ does not split in $K$,
    and $\OO$ an order in $K$ such that $p$ does not divide the conductor of $\OO$.
    Then the ideal class group $\mathrm{cl}(\OO)$ acts freely and transitively on $\rho(\Ell(\OO))$.
\end{theorem}

Thus, it follows from well-known results about
imaginary quadratic class numbers~\cite{siegel1935} that
asymptotically, there are $h(-D)\in\Omega(D^{\frac12-\varepsilon})$
many backdoor elliptic curves \emph{counted with
multiplicities} given by the number of embeddings of~$\OO$.
However, it is not generally clear that this corresponds
to many distinct \emph{curves}
(or maximal orders).
As an (extreme) indication of what could go wrong,
consider the following:
there seems to be no obvious reason why in some cases the
entire orbit of the group action of Theorem~\ref{thm:onuki}
should \emph{not} consist only of
one elliptic curve with lots of
independent copies of $\OO$ in its endomorphism ring.

\iffalse
We discuss two cases:
\fi
We can however at least prove that this does not always happen.
In fact, in the case that $D$ is small enough relative to $p$,
one can show that there cannot be more than one embedding
of $\OO$
into any maximal order in $B_{p,\infty}$,
implying that the $h(-D)$ oriented supersingular elliptic curves indeed
must constitute $h(-D)\approx\sqrt D$ distinct quaternion maximal orders:

\begin{lemma}
    \label{lemma:kaneko}
    Let $\mathcal O$ be a maximal order in $\BB_{p,\infty}$.
    If $D\equiv 3,0\pmod4$ is a positive integer smaller than $p$,
    then there exists at most one copy
    of the imaginary quadratic order of discriminant $-D$
    inside $\mathcal O$.
\end{lemma}
\begin{proof}
    This follows readily from Theorem~2$'$ of \cite{kaneko1989supersingular}.
\end{proof}

\noindent
This lemma together with Theorem~\ref{trapdoor} shows that there are
$\Theta(h(-D))$ many $(A,B)$-backdoor maximal orders under the restrictions that
$B > A^2$ and $D < p$.
Consider the case (of interest) in which $AB \approx p$: 
Following the same line of reasoning as in the proof of Theorem~\ref{trapdoor}
we have that 
$B^2/A^2 - A^2  \approx D$,
which if $D < p \approx AB$ implies that $B \lessapprox A^3$.
Hence, as advertised above, Lemma~\ref{lemma:kaneko} suffices to prove that there are
$\Theta(h(-D))$ many $(A,B)$-backdoor maximal orders under the restriction
that $AB \approx p$ and roughly $A^2<B<A^3$.
For larger choices of $\Ntwo$, it is no longer true that there is
only one embedding of $\OO$ into a quaternion maximal order:
indeed, at some point $h(-D)$ will exceed
the number~$\Theta(p)$ of available maximal orders,
hence there must be repetitions.
While it seems hard to imagine cases where the orbit of
$\cl(\Z[\theta])$ covers only a negligible number of curves
(recall that $\theta$ was our endomorphism of reduced norm $D$),
we do not currently know how
(and under which conditions)
to rule out this possibility.

\begin{remark}
Having obtained any one maximal order
$\mathfrak O$
that contains $\theta$,
it is efficient to compute more such orders
(either randomly or exhaustively):
For any ideal $\mathfrak a$ in~$\Z[\theta]$,
another maximal order with an optimal embedding of $\Z[\theta]$
is the right order of the left ideal $\mathcal I=\mathfrak O\mathfrak a$.
(One way to see this: $\mathfrak a$ defines a horizontal isogeny
with respect to the subring~$\OO$; multiplying by the full endomorphism
ring does not change the represented kernel subgroup; the codomain of
an isogeny described by a quaternion left ideal has endomorphism ring
isomorphic to the right order of that ideal. Note that this is similar
to a technique used by~\cite{rational-irrational} in the
context $\OO\subseteq\Q(\pi)$.)
\end{remark}

\subsection{Backdoored \texorpdfstring{$\bm{p}$}{TEXT} for given 
\texorpdfstring{$\bm{A}$}{TEXT} and \texorpdfstring{$\bm{B}$}{TEXT} with starting vertex \texorpdfstring{$\bm{j=1728}$}{TEXT}}
\label{ss:Insecure p}

Another way of constructing backdoor instances of an SIDH-style key exchange is to keep the starting vertex as $j=1728$ (or close to it),
keep $A$ and $B$ smooth or powersmooth
(but not necessarily only powers of $2$ and $3$ as in SIKE),
and construct the base-field prime $p$ to turn $j=1728$ into an $(A,B)$\nobreakdash-backdoor curve.
In this section, let $E_0$
denote the curve $E_0\colon\, y^2=x^3+x$.

An easy way of constructing such a $p$
is to perform steps~\ref{insecure_e} and~\ref{d}
of Algorithm~\ref{alg:insecure},
and then let $D := \frac{B^2e - d^2}{A^2}$.  
Then we can solve
\[
D = p(a^2+b^2) + c^2
\]
in variables $a,b,c,p \in \mathbb{Z}$, $p$ prime, as follows.
Factor $D-c^2$ for small $c$ until the result is of the form $pm$ where $p$ is a large prime congruent to $3$ modulo $4$ and~$m$ is a number representable as a sum of squares.\footnote{Some choices of $A$ and $B$ result in $D \equiv 2 \pmod 4$ which is an obstruction to this method.}

Then, with $\theta = a\iota\pi + b\pi + c\iota$
the tuple $(E_0,\theta,d,e)$ is $(A,B)$-backdoor.
(Note that, in this construction, we cannot expect to satisfy a relationship such as $p = ABf-1$ with small $f \in \Z$.)

As an (unbalanced) example, let us choose $A=2^{216}$ and $B=3^{300}$ and set $e=1$. Then we can use $d=B\bmod A^2$. Let $D=\frac{B^2-d^2}{A^2}$, for which we will now produce two primes:
First, pick $c=53$, then $D-c^2$ is a prime number (i.e., $a=1,~b=0$).
Second, pick $c=355$, then $D-c^2$ is 5 times a prime number (i.e., $a=2$, $b=1$).
Both of these primes are congruent to $3$ modulo $4$.

For a powersmooth example, let $A$ be the product of every other prime from $3$ up through $317$, and let $B$ be the product of all remaining odd primes $\le 479$.  With $e=4$, we can again use $d=B\bmod A^2$ and compute $D$ as above. Then $D - 153^2$ is prime and congruent to~$3$ modulo~$4$ (i.e., $a=1$, $b=0$).

\subsection{Backdoored $p$ for given \texorpdfstring{\bm{$A\approx B$}}{TEXT} with starting vertex \texorpdfstring{\bm{$j=1728$}}{TEXT}}\label{subsec:param}
For $A \approx B$, finding $(A,B)$-backdoor curves seems difficult. However, in this section we show that
certain choices of (power)smooth parameters $A$ and $B$ allow us to find $f$ such that $j=1728$ can be made insecure over any $\F_{p^2}$ with $p = ABf-1$.

One approach to this is to find Pythagorean triples $A^2 + d^2 = B^2$ where $A$ and $B$ are coprime and (power)smooth; then $E_0\colon y^2=x^3+x$ is a backdoor curve with $\theta=\iota$, the $d$ value from the Pythagorean triple, and $e=1$.  With this construction, we can then use \emph{any} $p\equiv3\pmod4$, in particular one of the form $p = ABf-1$. 

Note that given the isogeny degrees $A,B$, it is easy for anyone to detect if this method has been used by simply checking whether $B^2-A^2$ is a square; hence, an SIDH key exchange using such degrees is simply \emph{weak} and not just backdoored.%
\footnote{%
    We resist the temptation of referring to such instantiations as \enquote{door} instead of \enquote{backdoor}.
}

\begin{problem}
Find Pythagorean triples $B^2 = A^2 + d^2$ such that $A$ and $B$ are coprime and smooth (or powersmooth).
\end{problem}

Pythagorean triples can be parameterized in terms of Gaussian integers.  To be precise, primitive integral Pythagorean triples $a^2=b^2 + c^2$ are in bijection with Gaussian integers $z = m+ni$ with $\gcd(m,n)=1$ via the correspondence $(a,b,c) = \big({\nor(z)}, \real(z^2), \ima(z^2)\big)$.  The condition that $m$ and $n$ are coprime is satisfied if we take $z$ to be a product of split Gaussian primes, i.e., $z = \prod_i w_i$ where $\nor(w)~\equiv~1~\pmod 4$ is prime, taking care to avoid simultaneously including a prime and its conjugate.  Thus the following method applies provided that $B$ is taken to be an integer divisible only by primes congruent to $1$ modulo $4$, and $B > A$.

In order to guarantee that $B = \nor(z)$ is powersmooth, one may take many small $w_i$.  In order to guarantee that $B$ is smooth, it is convenient to take $z = w^k$ for a single small Gaussian prime $w$, and a large composite power $k$.

It so happens that the sequence of polynomials $\real(z^k)$ in variables $n$ and $m$ (recall $z=n+mi$) factors generically into relatively small factors for composite $k$, so that, when $B^2 = A^2 + d^2$, we can expect that $A$ is frequently smooth or powersmooth.
In practice, running a simple search using this method, one very readily obtains example insecure parameters:
\begin{align*}
    B &= 5^{105} \\
    A &= 2^2 \cdot 11 \cdot 19 \cdot 29 \cdot 41 \cdot 59 \cdot 61 \cdot 139 \cdot 241 \cdot 281 \cdot 419 \cdot 421 \cdot 839 \cdot 2381 \cdot 17921 \\
    &\hspace{2em}\cdot 21001 \cdot 39761 \cdot 74761\cdot 448139 \cdot 526679 \cdot 771961 \cdot 238197121 \\
    d &= 3^2 \cdot 13 \cdot 79 \cdot 83 \cdot 239 \cdot 307 \cdot 2801 \cdot 3119 \cdot 3361 \cdot 3529 \cdot 28559 \cdot 36791 \cdot 53759 \\
    &\hspace{2em}\cdot 908321 \cdot 3575762705759 \cdot 23030958433523039 
\end{align*}

For this example, if we take $p = 105 AB - 1$, we obtain a prime
%\begin{align*}
%p = \frst 24856463957240174917475336115723935764352796770215973
%    \cont 16928749873639702654215091732189441012151282410578678
%    \cont 055967927207348111551254987716674804687499
%    \text,
%\end{align*}
which is $3$ modulo $4$. 
Note that here $B \approx 2^{244}$ and $A \approx 2^{238}$.
Many other primes can easily be obtained (replacing $105$ with $214$, $222$, etc).

\begin{remark}
    When choosing parameter sets to run B-SIDH~\cite{DBLP:conf/asiacrypt/Costello20}, 
    if the user is very unlucky, 
    they could hit an instance of such a weak prime. 
    With this in mind, it would be prudent to check that a given combination of $A$, $B$, and $p$ does not fall into this category before implementing such a B-SIDH instance.
\end{remark}

%% file: non-poly.tex
\section{Non-polynomial-time attacks}\label{sec:non-poly}

So far we focused on polynomial-time algorithms both for the starting curve $E_0$ with \jinvariant 1728 and for backdoor curves,
which required the integer $e$ occuring in the attack to be polynomial in $\log p$.
However, the attack still works when $e$ is bigger, with decent scaling behaviour.
Hence, we may (and will in this section) consider algorithms which are exponential-time, yet improve on the state of the art.
The best known classical and quantum attacks for retrieving an isogeny of degree $A$ take time $\Ostar(A^{\frac12})$;
recall that we discussed quantum claw-finding in Subsection~\ref{sec:prel:quantumcost}.
We will adapt both the dual and the Frobenius isogeny attacks of Section~\ref{sec:improvements} 
to allow for some brute-force in order to attack balanced parameters.
We will also adapt the definition of backdoor curves 
to include curves for which there exists an exponential dual isogeny attack that improves on the state of the art, thus increasing the pool of backdoor curves.

\subsection{Non-polynomial time dual isogeny attack for $E_0: y^2=x^3+x$}\label{subsec:nonpoly}

Recall from Section~\ref{sec:improvements} that the dual isogeny attack consists of a
\enquote{precomputation} phase and a \enquote{key-dependent} phase.
The precomputation phase
(Algorithm~\ref{alg:death-eqn})
was to find a solution to Equation~\eqref{the_eqnofdeath}\,---\,notably, this depends only on the parameters $(p,\None,\Ntwo)$ and not on the concrete public key under attack.
The \enquote{key-dependent} phase
utilized said solution to recover the secret isogeny
via Theorem~\ref{more efficient}
for a specific public key.
Our modifications to the dual isogeny attack
come in three independent guises,
and the resulting algorithm is shown in Algorithm~\ref{alg:hybrid}:

\begin{itemize}
    \item \ul{Precomputation phase}: %
        \vspace{.2ex}
		\begin{itemize}
			\item \textbf{Larger $\bm{d}$:}
    			When computing a solution to Equation~\eqref{the_eqnofdeath},
    			we fix $e$ and then try up to 
    			$A^{\delta}$ values for $d$ 
    			until the equation has solutions. 
    			This allows us to further relax the constraints between 
    			$A$, $B$, and $p$, 
    			at the price of an exhaustive search of classical cost 
                $\Ostar(A^{\delta})$
                or quantum cost $\Ostar(A^{\frac{\delta}{2}})$
                using Grover's algorithm.
		\end{itemize}
    \vspace{.7ex}
	\item \ul{Key-dependent phase}: %
        \vspace{.2ex}
		\begin{itemize}
    		\item \textbf{Larger $\bm{e}$:}
    			We search for a solution to Equation~\eqref{the_eqnofdeath} 
    			where $e$ is any smooth number $\leq A^\epsilon$ 
    			with $\epsilon\in[0,1]$,
        		whereas in~\cite{Petit2017}
        		the integer
        		$e$ was required to be polynomial in $\log p$.
				This relaxes the constraints on $A$ and $B$, 
				at a cost of a 
				$\Ostar(e^{\frac{1}{2}}) = \Ostar(A^{\frac{\epsilon}{2}})$ 
				computation, both classically and quantumly, via 
				a meet-in-the-middle or claw-finding algorithm 
				(to retrieve the endomorphism 
				$\eta$ defined in the proof of
				Theorem~\ref{more efficient}).
    		\item \textbf{Smaller $\bm{A}$:}
    			We first naïvely guess part of the secret isogeny and
    			then apply the dual isogeny attack
    			only on the remaining part for each guess; see Figure~\ref{fig:smallerA}.
    			More precisely,
			    we iterate through isogenies of degree 
    			$A^\gamma \mid A$, 
    			with $\gamma \in [0,1]$, 
    			and for each possible guess we apply the dual isogeny attack to~\ref{torsion} with 
    			$A':= A^{1-\gamma}$ 
    			in place of $A$.
                The Diophantine equation to solve thus turns into
                \begin{equation}
                    \label{eqn:generalized_eqnofdeath}
                    \None'^2(p a^2 + p b^2 + c^2 )+ d^2 =  \Ntwo^2 e
                    \,\text.
                \end{equation}
                The cost of using $A'$ in place of $A$ is the cost of iterating over 
                the isogenies of degree $A^{\gamma}$
                multiplied by the cost $T$ of running the dual isogeny attack (possibly adapted as above to allow for larger $e$).
                This is an exhaustive search of cost $\Ostar(A^{\gamma}\cdot T) = \Ostar(A^{\gamma + \frac{\epsilon}{2}})$ classically
                or $\Ostar(A^{\frac{\gamma}{2}}\cdot T) = \Ostar(A^{\frac{\gamma+\epsilon}{2}})$ quantumly using Grover's algorithm.\footnote{
                For the reader who is wondering exactly how to apply Grover's algorithm in this context:
                Let $\langle P_A, Q_A \rangle = E_0[A^{\gamma}].$ 
                The input for Grover's algorithm here is an integer $n < A^{\gamma}$ and all of the input of Algorithm~\ref{alg:find_isog}.
                Attempt Steps~\ref{step2} and \ref{step3} for $\varphi_g$ such that $\ker(\varphi_g) = \langle P_A + nQ_A\rangle $;
                the output will be success or failure.
                Every subroutine of Steps~\ref{step2} and \ref{step3} can be broken down into basic elliptic curve arithmetic for which there are known quantum algorithms of similar complexity to their classical counterparts.
                }
                
		\end{itemize}
        \vspace{-2ex}
\end{itemize}

\begin{center}
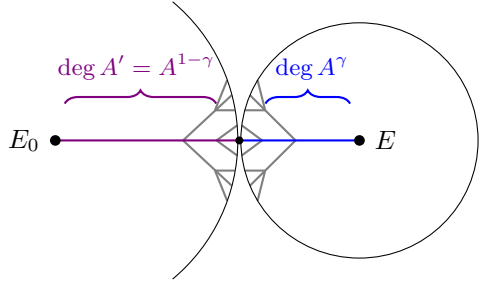
\begin{figure}
\hspace{4 cm}
    \begin{tikzpicture}[scale=0.4 ]
    \node (t0) at (0,1) {};
    \node (t1) at  (5.6,1) {};
\draw[decorate,decoration={brace, raise=2.5pt,amplitude=5pt}, thick,violet] (t0) -- node[above=8pt] 
{\begin{footnotesize}$\deg A'=A^{1-\gamma}\, \,$\end{footnotesize}} (t1);
    \node (t2) at (6.6,1) {};
    \node (t3) at  (10,1) {};
\draw[decorate,decoration={brace, raise=2.5pt,amplitude=5pt}, thick,blue] (t2) -- node[above=8pt] 
{\begin{footnotesize}$\,\, \deg A^{\gamma}$\end{footnotesize}} (t3);
        \draw[thick, color=gray] (5.3,0) -- (5.98,.5);
        \draw[thick, color=gray] (5.3,0) -- (5.98,-.5);
        \draw[thick, color=gray] (4.2,0) -- (5.8,-1.5);
        \draw[thick, color=gray] (5.25,-1) -- (5.9,-1);     
        \draw[thick, color=gray] (5.25,-1) -- (5.65,-2);
        \draw[thick, color=gray] (4.2,0) -- (5.8,1.5);
        \draw[thick, color=gray] (5.25,1) -- (5.9,1);     
        \draw[thick, color=gray] (5.25,1) -- (5.65,2); 
        \draw (6,0) arc (0:-50:6);
        \draw (6,0) arc (0:50:6);
        %next graph  
        \draw[thick, color=gray] (6.8,0) -- (6.13,.5);
        \draw[thick, color=gray] (6.8,0) -- (6.13,-.5);
        \draw[thick, color=gray] (7.9,0) -- (6.4,-1.5);
        \draw[thick, color=gray] (6.9,-1) -- (6.23,-1);     
        \draw[thick, color=gray] (6.65,-2) -- (6.9,-1);
        \draw[thick, color=gray] (7.9,0) -- (6.4,1.5);
        \draw[thick, color=gray] (6.9,1) -- (6.23,1);     
        \draw[thick, color=gray] (6.65,2) -- (6.9,1);        
        \draw[thick, color=blue] (10,0) -- (6.05,0);
        \draw[thick, color=violet] (0,0) -- (6.05,0) node[right] {{}};
        \node[circle,fill=black,inner sep=0pt,minimum size=4pt,label=left:{$E_0$}] (C1) at (0,0) {};
        \node[circle,fill=black,inner sep=0pt,minimum size=4pt,label=right:{$E$}] (C2) at (10,0) {};  
        \node[circle,fill=black,inner sep=0pt,minimum size=3pt,label=above:{}] (A) at (6.05,0) {};     
        %\draw (C1) circle (6);
        \draw (C2) circle (3.9);
        \end{tikzpicture}
\caption{Brute-force guessing the degree $A^{\gamma}$ part of Alice's isogeny $\phi$ from Alice's curve $E$ and the dual isogeny attack to find the remaining degree
$A'$ part of $\phi$ from $E_0$.}
\label{fig:smallerA}
\end{figure}
\end{center}

\begin{algorithm}
    \caption{\,Solving the norm equation; precomputation.} \label{alg:solve_eqn}
    \vspace{.2ex}
    \Input {%
            \mybullet{} SIDH parameters $p,A=p^\alpha,B=p^\beta$.
            \InputNewLine
            \mybullet{} Attack parameters $\delta,\gamma,\epsilon\in[0,1]$, with $A^\gamma\mid A$.
        }
    \Output {%
        A solution $(a,b,c,d,e)$
        to \eqref{eqn:generalized_eqnofdeath} with $A'=A^{1-\gamma}$
        and $e\leq A^\epsilon$ smooth.
        }
    %
    % precomputation
    %
    \vspace{.4ex}
    \label{step:start}%
           Pick a smooth number $e\leq A^\epsilon$ which is a square modulo $A'^2$.
           \;
           Compute $d_0$ such that $d_0^2\equiv eB^2\pmod{A'^2}$.
           \;
    \For{$d'=1,2,...,\lfloor A^\delta\rfloor$ such that $d_0+A'^2d' < \sqrt e B$} {
        \vspace{.2ex}
	    Let $d=d_0+A'^2d'$.
        \;\label{step:find-c}
        Find the smallest positive integer $c$ such that $c^2A'^2=eB^2-d^2\pmod{p}$,
            or~\textbf{continue}~if~no~such~$c$~exists.
        \;
        \vspace{.4ex}
        \label{step:secondif}
        \If{$eB^2>d^2+c^2A'^2$} {
               \label{step:endeqdeath}
               \label{step:cornacchia}
               Try finding $(a,b)$ such that $a^2+b^2=\frac{eB^2-d^2-c^2A'^2}{A'^2p}$.
               If~a~solution~is~found,~\textbf{return}~$(a,b,c,d,e)$.
           }
       }
\end{algorithm}

\begin{algorithm}
    \caption{\,Recovering the secret isogeny; key-dependent phase.} \label{alg:find_isog}
    \vspace{.2ex}
    \Input {%
            \mybullet{} All the inputs of Algorithm~\ref{alg:solve_eqn}.
            \InputNewLine
            \mybullet{} An instance of~\ref{torsion} with those parameters, namely a curve $E$ and points $P,Q\in E[B]$ where there exists a degree-$A$ isogeny $\varphi:E_0\rightarrow E$ such that $P,Q$ are the images by $\varphi$ of a canonical basis of $E_0[B]$.
            \InputNewLine
            \mybullet{} $\theta\in\End(E_0)$ and $d,e\in\Z$ such that $\deg(A'\theta+d)=B^2e$ with $e\leq A^{\epsilon}$ smooth.
        }
    \Output {An isogeny $\varphi$ matching the constraints given by the input.}
    \vspace{.7ex}
    \label{step:guessisog}
    \For{$\varphi_g:E\rightarrow E'$ an $A^\gamma$-isogeny} {
        \vspace{.4ex}
	    Compute $P'=[A^{-\gamma}\bmod{B}]\,\varphi_{g}(P)$ and $Q'=[A^{-\gamma}\bmod{B}]\,\varphi_{g}(Q)$.
        \label{step2}\;
        \vspace{.2ex}
	    Use Theorem~\ref{more efficient} to compute $\varphi':E_0\rightarrow E'$ of degree $A' = A^{1-\gamma}$,
	    assuming~that~$P'$~and~$Q'$~are~the~images~by~$\varphi'$
	    of~the~canonical~basis~of~$E_0[B]$,
	    or~conclude~that~no~such~isogeny~exists.
        \label{step3}\;
        \vspace{.4ex}
        \If{$\varphi'$ is found}{
            \vspace{.2ex}
            \Return {$\varphi=\dual{\varphi_g}\circ\varphi'$.} \label{step:part2end}
        }
    }
\end{algorithm}

\begin{algorithm}
    \caption{\,Solving~\ref{torsion}.} \label{alg:hybrid}
    \vspace{.4ex}
        Invoke Algorithm~\ref{alg:solve_eqn},
        yielding $a,b,c,d,e\in\Z$,
        and then Algorithm~\ref{alg:find_isog}
        with $\theta=a\iota\pi+b\pi+c\iota$.
\end{algorithm}

\begin{proposition}\label{prop:nonpolydual}
Define $\alpha$ and $\beta$ by setting $A = p^{\alpha}$ and $B = p^{\beta}$ and fix $0<\alpha\leq\beta$.
Under Heuristic~\ref{heur:deathattack-kate}, 
if 
    \[
        2\beta+\alpha\epsilon
        \hspace{.4em} \geq \hspace{.4em}
        \max\,\{4\alpha+2\alpha\delta-4\alpha\gamma,\ 2+2\alpha-2\alpha\delta-2\alpha\gamma\}
        \,\text,
    \]
$A$ has (at most) $O(\log\log p)$ distinct prime factors, and $B$ is at most polynomial in $A$, 
then Algorithm~\ref{alg:hybrid} solves~\ref{torsion} in time $\Ostar(A^{\Gamma})$
on a classical computer and time $\Ostar(A^{\Gamma/2})$ on a quantum computer, where
    \[
        \Gamma
        \hspace{.2em}
        :=
        \hspace{.2em}
        \max\Big\{ \frac{1+3\alpha-2\beta}{3\alpha},\, \frac{2\alpha-\beta}{2\alpha},\, \frac{1+\alpha-\beta}{2\alpha} \Big\}
        \,\text.
    \]
    \CP{and negligible memory cost}
\end{proposition}

\begin{proof}
    \iffinal
    See the full version~\cite[Appendix~A.2]{fullversion}.
    \else
    See Appendix~\ref{proof-nonpolydual}.
    \fi
\qed
\end{proof}

\begin{corollary}\label{cor:nonpolydual}
Suppose that $B$ is at most polynomial in $A$ and that $A$ has (at most) $O(\log\log p)$ distinct prime factors.
When run on a classical computer, 
Algorithm~\ref{alg:hybrid} is asymptotically more efficient than meet-in-the-middle\,---\,disregarding memory concerns, so more efficient than $\Ostar(A^\frac{1}{2})$\,---\,whenever 
\[B > \max\left\{
\sqrt{p}A^{\frac{3}{4}}, A, p
\right\}.\]
When run on a quantum computer,
Algorithm~\ref{alg:hybrid} is asymptotically more efficient than quantum claw-funding\,---\,according to the model in~\cite{DBLP:conf/crypto/JaquesS19}, so more efficient than $\Ostar(A^{\frac{1}{2}})$\,---\,whenever 
\[B > \max\left\{
\sqrt{p}, A^{-1}p
\right\}.\]
\end{corollary}

\subsection{Non-polynomial time Frobenius isogeny attack for $E_0: y^2 = x^3 + x$}\label{sec:nonpolyfrob}

Recall the Frobenius isogeny attack from Section~\ref{sec:improvements}.
In a similar way to the previous section, we allow for some brute-force to improve the balance of our parameters.
More precisely, we consider again:
\begin{itemize}
    \item[$\bullet$] \textbf{Smaller $A$:} Iterate through isogenies of degree $A^{\gamma} |A$; 
    in the precomputation we solve instead
    \begin{equation}\label{eqn_ofevenlessdeath}
        A'^2(a^2 + b^2) + pc^2 = B^2e,
    \end{equation}
    where $A' = A^{1-\gamma}$.
\end{itemize}
\noindent
Algorithm~\ref{alg:nonpoly_frob_new} describes how to adapt the Frobenius isogeny attack of Section~\ref{sec:improvements} in this way.
\vspace{2ex}

\begin{algorithm} [H] 
    \caption{\,Solving~\ref{torsion}.} \label{alg:nonpoly_frob_new}
    \vspace{.4ex}
        (Precomputation) Invoke Algorithm~\ref{alg:frob-eqn} with inputs $p,A',B$,
        yielding $a,b,c,e\in\Z$.\\
        (Key-dependent) Run Algorithm~\ref{alg:find_isog}
        except that $\theta=a\iota\pi+b\pi \in\End(E_0)$ instead satisfies the equation $\deg(A'\theta+c)=B^2ep$ and we use Theorem~\ref{frob-thm-kate} in place of
        Theorem~\ref{more efficient}.
\end{algorithm}

\begin{proposition}\label{prop:nonpolyfrob}
Define $\alpha$ and $\beta$ by $A = p^{\alpha}$ and $B = p^{\beta}$, fix $B \geq A$ and $B$ at most polynomial in $A$, and suppose that $A' = A^{1-\gamma}$ has (at most) $O(\log\log p)$ distinct prime factors.
Under Heuristic~\ref{heur:deathattack-kate2},
Algorithm~\ref{alg:nonpoly_frob_new} has complexity
 $\Ostar\left(A^{\gamma}\right) = \Ostar\left( A^{ \frac{1+4\alpha -2\beta}{4\alpha}}\right)$
    classically and $\Ostar\left(A^{\frac{\gamma}{2}}\right) = \Ostar\left( A^{ \frac{1+4\alpha -2\beta}{8\alpha}}\right)$ quantumly.
    Moreover, the precomputation step runs in time $\Ostar(1)$.
\end{proposition}

\begin{proof}
    \iffinal
    See the full version~\cite[Appendix~A.3]{fullversion}.
    \else
    See Appendix~\ref{proof-nonpolyfrob}.
    \fi
\end{proof}

\begin{corollary}\label{cor:nonpolyfrob}
Suppose that $B$ is at most polynomial in $A$ and that $A$ has (at most) $O(\log\log p)$ distinct prime factors.
When run on a classical computer, 
Algorithm~\ref{alg:nonpoly_frob_new} is asymptotically more efficient than meet-in-the-middle\,---\,disregarding memory concerns, 
so more efficient than $\Ostar(A^\frac{1}{2})$\,---\,whenever $B > \sqrt{p}A$.
When run on a quantum computer,
Algorithm~\ref{alg:nonpoly_frob_new} is asymptotically more efficient than quantum claw-funding\,---\,according to the model in~\cite{DBLP:conf/crypto/JaquesS19}, so more efficient than $\Ostar(A^{\frac{1}{2}})$\,---\,whenever $B > \sqrt{p}$.
\end{corollary}

\begin{remark}
It may seem natural to also allow for larger $e$ as in the dual isogeny attack. However, this limits how small $A'$ can be, and the gain from reducing $A'$ is strictly better than the gain from increasing $e$. 
Intuitively this is because $A'$ appears in Equation~\ref{eqn_ofevenlessdeath} as a square, which doubles the gain compared to gain from increasing $e$.
\end{remark}

\subsection{Non-polynomial time dual isogeny attack for backdoor curves}\label{subsec:nonpolytrap}

Recall the definition of an $(A,B)$-backdoor curve $(E_0,\theta,d,e)$ 
from Definition~\ref{insecure};
we now extend this to define backdoor curves that give rise to a torsion-point attack of complexity $\Ostar(A^\comp)$.
In this section we explain how to modify Algorithm~\ref{alg:insecure} to compute these more general backdoor curves, and
apply Algorithm~\ref{alg:hybrid} with such a backdoored starting curve $E_0$ by replacing the precomputation step with the modified Algorithm~\ref{alg:insecure}.

\begin{definition}\label{insecure-comp}
Let $A,B$ be coprime positive integers and $0\leq\comp\leq 1/2$.
    An \emph{$(A,B,\comp)$-backdoor curve} is a tuple $(E_0,\theta,d,e)$
    of a supersingular elliptic curve $E_0$ over some $\F_{p^2}$,
    an endomorphism $\theta\in\End(E_0)$ in an efficient representation,
    and two integers~$d,e$,
    such that
    Algorithm~\ref{alg:find_isog}
    solves~\ref{torsion}
    for that particular~$E_0$
    in time~$\Ostar(A^{\comp})$
    when given~$(\theta,d,e)$.
    An $(A,B,0)$-backdoor curve
    is then an $(A,B)$-backdoor curve in the sense of Definition~\ref{insecure}.
\end{definition}

\noindent
To construct $(A,B,\comp)$-backdoor curves, we modify Algorithm~\ref{alg:insecure} as follows:
\begin{itemize}
    \item Use
        $A'=A^{1-\gamma}$ instead of $A$, namely we will guess part of the isogeny with degree $A^\gamma \mid A$.
    \item Instead of starting from $e=1$, choose $A^{\epsilon'}$ random values of $A'^4 B^{-2} < e
         \leq A^{\epsilon}$ 
        (note $e$ is not necessarily an integer square) 
        until there exists $d$ such that $d^2=B^2e\bmod (A')^2$,
        \begin{equation}\label{sign-cond}
            B^2e-d^2 > 0
        \end{equation}
        and $B^2e-d^2$ is a square modulo $p$. 
        Once these values of $d$ and $e$ are found, 
        continue like in 
        Algorithm~\ref{alg:insecure}, Step~\ref{step:solve}.
\end{itemize}

\begin{proposition}\label{extend-improve}
    Heuristically, if $A$ has (at most) $O(\log\log p)$ distinct prime factors:
\begin{itemize}
    \item Let $\comp \in [0,0.4]$.
    For $A$, $B$ such that $B > A^{2 - \frac{5}{2}\cdot\comp}$,
    Algorithm~\ref{alg:insecure} modified as above
    constructs a $(A,B,\comp)$-backdoor curve
    in time $\Ostar(A^\comp)$ on a classical computer,
    assuming an oracle for factoring.
    \item Let $\comp \in [0,0.25]$.
    For every $A$, $B$ such that $B > A^{2-4\cdot\comp}$,
    Algorithm~\ref{alg:insecure}
    constructs a $(A,B,\comp)$-backdoor curve
    in polynomial time on a quantum computer.
\end{itemize}
\end{proposition}

\begin{proof}
    \iffinal
    See the full version~\cite[Appendix~A.4]{fullversion}.
    \else
    See Appendix~\ref{proof-nonpolytrapdoor}.
    \fi
\end{proof}

\begin{corollary}
When $A \approx B$ (e.g. as in SIKE \cite{azarderakhsh2017supersingular}), the modified Algorithm~\ref{alg:insecure}
computes a $(A,B,\frac{2}{5})$-insecure curve in time $\Ostar(A^{\frac{2}{5}})$
on a classical computer
and computes a $(A,B,\frac{1}{4})$-insecure curve in polynomial time on a quantum computer.
In particular, when $A \approx B \approx \sqrt{p}$, 
there exist backdoor curves $E_0$ for which we can solve~\ref{torsion} on a classical computer in time $\Ostar(p^\frac{1}{5})$ and for which we can solve~\ref{torsion} on a quantum computer in time $\Ostar(p^\frac{1}{8})$.
\end{corollary}

\CP{missing a section on Frobenius attack for backdoor curves?}

%% file: impact.tex
\section{Impact on unbalanced SIDH, group key agreement, and B-SIDH}\label{sec:impact}

We summarize how the results of Sections \ref{sec:improvements}, \ref{subsec:nonpoly}, and \ref{sec:nonpolyfrob} impact unbalanced SIDH with $p \approx AB$, 
the GSIDH multiparty group key agreement~\cite{DBLP:conf/isita/FurukawaKT18, DBLP:journals/iacr/AzarderakhshJJS19},
and 
B-SIDH~\cite{DBLP:conf/asiacrypt/Costello20}.

\subsection{Frobenius isogeny attack on group key agreement and unbalanced SIDH}\label{ss:frob_nonpoly}

Let us consider unbalanced SIDH with $p \approx AB$.
More precisely, we study instances of~\ref{torsion} with $p = AB\cdot f-1$, 
where $f$ is a small cofactor and where $A$ has (at most) $O(\log\log p)$ distinct prime factors.
Then by Proposition~\ref{prop:eq2} and Theorem~\ref{frob-thm-kate},
under Heuristic~\ref{heur:deathattack-kate2},
the Frobenius isogeny attack of Section~\ref{sec:improvements} 
gives a polynomial-time attack on~\ref{torsion} when $B > \sqrt{p}A^2$.
Since in this section we restrict to the case $p \approx AB$, this inequality simplifies to
$B \geq A^5$. In particular, this gives us one of our main results:

\begin{theorem}\label{thm:polybreakgroupkey}
  Under Heuristic~\ref{heur:deathattack-kate2}, the Frobenius isogeny attack presented in Section~\ref{sec:improvements} breaks the GSIDH $n$-party group key agreement protocol presented
  in~\cite{DBLP:conf/isita/FurukawaKT18,DBLP:journals/iacr/AzarderakhshJJS19} in time polynomial in $\log p$ for all $n\geq 6$.
\end{theorem}

\begin{proof}
Recall from Subsection~\ref{subsec:sidh} that
the cryptanalytic challenge underlying the $n$-party group key  agreement as presented in \cite{DBLP:conf/isita/FurukawaKT18,DBLP:journals/iacr/AzarderakhshJJS19}
can be modelled as an instance of \ref{torsion} with $A = \ell_1^{e_1}$, $B = \ell_2^{e_2}\cdots\ell_n^{e_n}$, and $p = AB\cdot f - 1$, where $\ell_1,\ldots,\ell_n$ are primes such that for all $i,j$ we have $\ell_i^{e_i} \approx \ell_j^{e_j}$ and $f$ is a small cofactor chosen such that $p$ is prime.
Thus the security of the $n$-party group key agreement is similar
to that of unbalanced SIDH with the same $p,A,B$.
Suppose $n\geq 6$. Since $A$ is a prime power (hence has $1=\bigO(\log\log p)$ prime divisors) and $B \geq A^5$, 
the Frobenius isogeny attack on the group key  agreement is polynomial-time
when there are 6 or more parties.
\qed
\end{proof}

We have implemented this attack in Magma~\cite{bosma1997magma} for 6 parties, see the code at~\url{https://github.com/torsion-attacks-SIDH/6party}.
The code is written to attack the power-of-3 torsion subgroup, when $p+1$ is powers of the first 6 primes, and uses cryptographically large parameters.

We know the Frobenius isogeny attack is polynomial on unbalanced SIDH when $B \geq A^5$ 
(and the $n$-party group key agreement when $n\geq 6$);
it remains to investigate the non-polynomial analogue.
To this end, consider the attack presented in Subsection~\ref{sec:nonpolyfrob}.
As above, suppose given an instance of~\ref{torsion} with $p = AB\cdot f-1$, 
where $f$ is a small cofactor, such that $A$ has (at most) $O(\log\log p)$ distinct prime factors,
and now additionally suppose that $B \approx A^{1 + \epsilon}$, where $0 < \epsilon < 4$.
To apply this attack to $n$-party group key agreement with $n=2,3,4,5$, just set~$\epsilon = n-2$.

Proposition~\ref{prop:nonpoly_unbalanced} demonstrates an improvement on the asymptotic complexity for quantum claw-finding as analyzed in~\cite{DBLP:conf/crypto/JaquesS19}
for any level of imbalance (i.e., for any $\epsilon > 0$).
However, 
note that the only quantum subroutine used in our Frobenius isogeny attack is Grover's algorithm, 
so our
complexity computation is independent of the choice of quantum computation model used for claw-finding. As such, using a more nuanced model working with concrete complexities, such as the one presented in \cite{DBLP:journals/iacr/JaquesS20},
will make our quantum attack start to \enquote{improve on the state of the art} at different levels of imbalance.
As our work currently only presents asymptotic complexities, we are leaving an analsyis of this for future work.

\begin{proposition}\label{prop:nonpoly_unbalanced}
Let $A,B$ be coprime smooth numbers where
$B>A^{1+\epsilon}$, and let $p$ be a prime congruent to $3 \pmod 4$. 
Furthermore, suppose that $p=ABf-1$ for some small cofactor $f$, 
and that the number of distinct prime factors of $A$ is (at most) $O(\log\log p)$. 
Let $E_0/\F_p$ be the supersingular elliptic curve with \jinvariant 1728.
Algorithm~\ref{alg:nonpoly_frob_new} solves~\ref{torsion} with these parameters in time
$\Ostar\left(A^{1-\frac{\epsilon }{4}}\right)$ when run on a classical computer and
time $\Ostar\left( A^{\frac{1}{2}-\frac{\epsilon }{8}}\right)$ when run on a quantum computer.
\end{proposition}

\begin{proof}
Let $\alpha = \frac{1}{2+\epsilon}$ and $\beta = \frac{1+\epsilon}{2+\epsilon}$. Proposition~\ref{prop:nonpolyfrob}
proves that Algorithm~\ref{alg:nonpoly_frob_new} runs classically in time
$$
\Ostar\left( A^{ \frac{1+4\alpha -2\beta}{4\alpha}}\right)=\Ostar\left(A^{\frac{(2+\epsilon)+4-2(1+\epsilon)}{4}}\right)
=\Ostar \left(A^{1-\frac{\epsilon }{4}}\right).
$$
Similarly, Proposition~\ref{prop:nonpolyfrob}
proves that Algorithm~\ref{alg:nonpoly_frob_new} runs quantumly in time
$$
\Ostar\left( A^{ \frac{1+4\alpha -2\beta}{8\alpha}}\right)=\Ostar\left(A^{\frac{(2+\epsilon)+4-2(1+\epsilon)}{8}}\right)
=\Ostar \left(A^{\frac{1}{2}-\frac{\epsilon }{8}}\right)
\text.
$$
\par\vspace{-3ex}\qed
\end{proof}

As stated above, akin to the polynomial-time attack, 
substituting $n=\epsilon+2$ in Proposition~\ref{prop:nonpoly_unbalanced} gives us the complexity of the non-polynomial Frobenius isogeny attack on $n$-party group key agreement for $n = 2,3,4,5$ parties, see
Table~\ref{table:group-key-exchange}.

\begin{table}[!ht]
    \caption[]{%
    Asymptotic complexities of our Frobenius isogeny attack on $n$-party key agreement and comparison with the state of the art, i.e., meet-in-the-middle and claw-finding.\protect\footnotemark{}
    Numbers given are the logarithm to base $A$ of the complexity, ignoring factors polynomial in $\log p$.}
    \label{table:group-key-exchange}
\centering
    \begin{tabular}{c||c|c|c|c}
        \toprule
         \# parties ~&~This work (classical) ~&~This work (quantum) ~&~MitM (classical) ~&~\cite{DBLP:conf/crypto/JaquesS19} (quantum) \\
         \midrule
         $\phantom{\geq\,}2$ & $1$ & ${1/2}$ && \\
         $\phantom{\geq\,}3$ & ${3/4}$ & ${3/8}$ & $\large|$ & $\large|$ \\
         $\phantom{\geq\,}4$ & ${1/2}$ & ${1/4}$ & $1/2$ & $1/2$ \\
         $\phantom{\geq\,}5$ & ${1/4}$ & ${1/8}$ & $\large|$ & $\large|$ \\
         ${\geq\,}6$ & $0$ & $0$ & & \\
         \bottomrule
    \end{tabular}
\end{table}
\footnotetext{%
    As justified above, we take~\cite{DBLP:conf/crypto/JaquesS19} for the \enquote{state-of-the-art} numbers for quantum claw-finding here rather than~\cite{DBLP:journals/iacr/JaquesS20}.}

\subsection{Dual isogeny attack applied to B-SIDH}
\label{impact-BSIDH}

A recent proposal called \emph{B-SIDH}~\cite{DBLP:conf/asiacrypt/Costello20} consists of instantiating SIDH with parameters where $AB$ is a divisor of $p^2-1$.
By Proposition~\ref{prop:nonpolydual}, under Heuristic~\ref{heur:deathattack-kate},
when $A \approx B \approx p$ (that is, $\alpha \approx \beta \approx 1$),
Algorithm~\ref{alg:hybrid} yields a quantum attack on these parameters
of complexity $\Ostar(A^{\frac13})=\Ostar(p^{\frac13})$.
This compares to other attack
complexities
in the literature as follows:
\begin{itemize}
    \item Tani's quantum claw-finding algorithm~\cite{tani}
        was claimed to have complexity~$\Ostar(p^{\frac13})$,
        but~\cite{DBLP:conf/crypto/JaquesS19} argues that the complexity
        is actually no lower than~$\Ostar(p^{\frac23})$
        when the cost of data-structure operations
        is properly accounted for.
    \item A quantum algorithm due to
        Biasse, Jao, and Sankar~\cite{BJS}
        finds \emph{some} isogeny between the start and end curve
        in time $\Ostar(p^{\frac14})$.
        While there is a heuristic argument for \enquote{standard} SIDH/SIKE
        that any isogeny suffices to find the correct isogeny~\cite{galbraith2016security},
        this argument relies on the fact that the isogeny sought in SIKE
        has relatively small degree compared to~$p$,
        which is was not believed to be true for B\nobreakdash-SIDH.
        The B\nobreakdash-SIDH paper~\cite{DBLP:conf/asiacrypt/Costello20}
        conservatively views \cite{BJS} as the best quantum attack.
        Since the publication of B-SIDH, it has been shown~\cite{takoisogeny}
        that~\cite{BJS} does in fact apply, so this is currently the best known quantum attack against B-SIDH.
    \item The cost of known classical attacks is no lower than~$\Ostar(A^{\frac12})$,
        which is achieved by meet-in-the-middle techniques (using exponential memory)
        and potentially memoryless by Delfs and Galbraith~\cite{delfs-galbraith}
        when $A\approx p$
        assuming a sufficiently efficient method to produce
        \emph{the} isogeny from some isogeny.
        \LP{does \cite{takoisogeny} not apply here?}
\end{itemize}

\noindent
Thus, assuming Heuristic~\ref{heur:deathattack-kate} holds,
Algorithm~\ref{alg:hybrid}
is
asymptotically
better than quantum claw-finding but is not the best known quantum attack against B\nobreakdash-SIDH
at the moment.

Note that for $1/2<\alpha\approx\beta<1$,
the
(quantum)
attack cost in terms of $p$
may be lower than $\Ostar(p^{\frac13})$,
but it does not get
smaller than $\Ostar(p^{\frac14})$ and hence does not improve on~\cite{BJS}
for $\alpha \approx \beta$.

\subsection{Impact on B-SIDH group key exchange}\label{sec:bsidhgroup}

As an example of how care should be taken when constructing new SIDH-style schemes, we also include a scheme that does not exist in the literature: group key agreement instantiated with B-SIDH parameters.
This is a natural scheme to consider: The size of the base-field prime used in group key agreement grows with the number of parties, and optimally chosen B-SIDH parameters (with respect to efficiency) halves the bit-length of the base-field prime.
Corollary~\ref{cor:bsidhgroupkey} shows that such an instantiation is insecure for 4 or more parties: 

\begin{corollary}\label{cor:bsidhgroupkey}
Let $A,B$ be coprime smooth numbers and let $p$ be a prime congruent to $3\pmod 4$. Furthermore, suppose that $p^2-1=ABf$ for some small cofactor $f$ and that 
$B>A^3$. Let $E_0$ be the supersingular elliptic curve with \jinvariant 1728.
Then, assuming Heuristic~\ref{heur:deathattack-kate},~\ref{torsion} can be solved in polynomial time.
\end{corollary}
\begin{proof}
The result follows from Proposition \ref{prop:eq1}.
\end{proof}
Finally, in Corollary~\ref{cor:bsidhgroupkeynonpoly} 
we give the complexity of our dual isogeny attack on an instantiation of B-SIDH 3-party group key agreement with minimal base-field prime:

\begin{corollary}\label{cor:bsidhgroupkeynonpoly}
Let $A,B$ be coprime smooth numbers and let $p$ be a prime congruent to $3\pmod 4$. Furthermore, suppose that $p^2-1=ABf$ for some small cofactor $f$ and that 
$B>A^2$. Let $E_0$ be the supersingular elliptic curve with \jinvariant 1728.
Then, assuming Heuristic~\ref{heur:deathattack-kate}, Algorithm~\ref{alg:hybrid} solves~\ref{torsion} in time $\Ostar(A^\frac14) = \Ostar(p^\frac16)$ when run on a classical computer and time $\Ostar(A^\frac18) = \Ostar(p^\frac{1}{12})$ when run on a quantum computer.
\end{corollary}

\begin{proof}
This follows from plugging $\alpha = 2/3$ and $\beta = 4/3$ into Proposition~\ref{prop:nonpolydual}.
\end{proof}

%% file: open-questions.tex
\section{Open Question}\label{sec:open}

The two attack variants given in Theorems~\ref{more efficient} and~\ref{frob-thm-kate}
may seem somewhat ad hoc at first.
In this Section, we describe a common abstraction for
both variants and discuss potential generalizations.

The core idea is to relax the choice of $\tau$
as an \textit{endo}morphism of $E$,
instead
allowing $\tau$
to be an isogeny
from $E$ to another curve $E'$:

\begin{theorem}
    \label{thm:general}
    Suppose given an instance of~\ref{torsion}
    where $A$ has $O(\log\log p)$ distinct prime factors.
    Let $\omega\colon E\to E'$
    be a known isogeny
    to some curve $E'$.
    Furthermore, assume we are
    given the restriction to $E_0[B]$
    of an isogeny $\psi\colon E_0\to E'$,
    and an integer $d\in\Z$
    such that
    the isogeny
    $\tau=\psi\dual\varphi+d\omega\in\Hom(E,E')$
    has degree $B^2e$,
    where $e$ is smooth.
    Then,
    we can compute a matching isogeny $\varphi$
    in time $\Ostar(\hspace{-.2em}\sqrt{e})$.
\end{theorem}
\begin{proof}
    The proof is completely analogous to Theorems~\ref{more efficient} and~\ref{frob-thm-kate}.
\end{proof}

\noindent
The specific instantiations obtained
as special cases
earlier
can be
recovered as follows:
\begin{itemize}[topsep=1.4ex,itemsep=1.4ex]
    \item For
        Theorem~\ref{more efficient},
        we simply use $E'=E$,
        the map $\omega$ is the identity morphism on $E$,
        and
        the isogeny $\psi$ is an
        element of the set
        $M'=\varphi M\subseteq \Hom(E_0,E)$,
        where ${M\leq\End(E_0)}$ 
        is the subgroup of
        trace-zero endomorphisms of $E_0$.
    \item For
        Theorem~\ref{frob-thm-kate},
        we use the Galois conjugate $E'=E^\sigma$ of $E$,
        the map
        $\omega\colon E\to E^\sigma$ is
        the
        $p$-power Frobenius isogeny,
        and
        the isogeny $\psi$ is an
        element of the set
        $M'=\varphi^\sigma M\subseteq \Hom(E_0,E^\sigma)$,
        where ${M\leq\End(E_0)}$
        is the subgroup orthogonal to
        Frobenius $\pi\in\End(E_0)$.%
        \footnote{%
        The way Theorem~\ref{frob-thm-kate}
        is presented differs from Theorem~\ref{thm:general} here;
        this is merely a change in notation.}
\end{itemize}

\noindent
In both cases, the choice of $M'$ and $\omega$
is such that the resulting degree form for
the subgroup $M'+\omega\Z$
of $\Hom(E_0,E')$
has a sufficiently nice shape to be solved efficiently
using techniques such as those shown in Subsection~\ref{sec:solving-norm}.

It is unclear whether there are any other choices of $M'$ and $\omega$
which lead to an
efficiently
solvable norm equation and potentially improved attacks.
However,
so far we have not found any other
ways
to exploit this viewpoint
beyond
using
$\varphi$
itself or its Galois conjugate.
Finding other useful generalizations is an interesting open problem.

%% file: proof.tex
\section{Proofs}\label{sec:proof}

In this section we will cover the proofs not contained in the main body of the paper.

\subsection{Proof of Lemma~\ref{lem:shift_lolli}}\label{proof-shift_lolli}

\begin{proof}[of Lemma~\ref{lem:shift_lolli}]
Subtracting $\bd$ from $\tau$ gives $ \rho= \phi \circ \theta \circ \widehat{\phi}$.
Once we know the lollipop endomorphism $ \rho$, then one can do the following. First one computes the intersection $\ker \rho \cap E_1[A]$ which can be accomplished efficiently as $A$ is smooth. 
If this intersection is cyclic, then $\ker \rho \cap E_1[A]=\ker(\widehat{\phi})$. 

If not, then one can do the following. Let $M$ be the largest integer such that 
$E_1[M]\subset (\ker\rho \cap E_1[A])$. Then one can decompose the secret isogeny $\phi$ as $\widehat{\phi_{A/M}}\circ\phi_M$ where the $\phi_M$ has degree $M$ and $\phi_{A/M}$ has degree $A/M$. 
In \cite[Subsection 4.3]{Petit2017} the following is shown:
\begin{enumerate}
    \item $\ker(\phi_{A/M})=M(\ker(\rho)\cap E_1[A])$ \label{eqn:find_phiAM}
    \item $\theta(\ker(\phi_M))=\ker(\phi_M)$
    \item  The number of subgroups of $E_0[M]$ fixed by $\theta$ is at most $2^{k}$,
    where $k$ is the number of distinct prime factors of $M$.
\end{enumerate}
Fact~\ref{eqn:find_phiAM} shows us how to compute $\phi_{A/M}$. 
Finally, $\phi_M$ can be computed by an exhaustive search 
(going through all the subgroups of $E_0[M]$ which are fixed by $\theta$). This last search is efficient by the condition that $A$ has $O(\log\log p)$ distinct prime factors. 
\qed
\end{proof}
\begin{remark}
In practice $M$ has usually $O(\log\log p)$ prime factors even without the extra condition on the degree of $\theta-[1]$. 
\end{remark}

\subsection{Proof of Proposition~\ref{prop:nonpolydual}}\label{proof-nonpolydual}

\begin{lemma}\label{lem:eqdeathsol}
Under Heuristic~\ref{heur:deathattack-kate}, if
    \begin{equation}\label{death-ineq}
        2\beta+\alpha\epsilon
        \hspace{.4em} \geq \hspace{.4em}
        \max\,\{4\alpha+2\alpha\delta-4\alpha\gamma,\ 2+2\alpha-2\alpha\delta-2\alpha\gamma\}
        \,\text,
    \end{equation}
then Algorithm~\ref{alg:solve_eqn} returns a solution $(a,b,c,d,e)$ to Equation~\eqref{eqn:generalized_eqnofdeath} with $A' = A^{1-\gamma}$ and $e \leq A^{\epsilon}$ smooth in time
$\Ostar(A^\delta)$ when run on a classical computer and in time
$\Ostar(A^{\frac{\delta}{2}})$ when run on a quantum computer.
\end{lemma}

\begin{proof}
To find an appropriate $e$ in Step~\ref{step:start},
find an integer square $e_0$ of size $\approx \log p$, 
of which there are approximately $1/\sqrt{\log p}$,
and take $e$ to be an even power of $e_0$ of size approximately $A^{\epsilon}$, modulo $A'^2$.
Then Step~\ref{step:start} takes time $\Ostar(1)$,
and $d_0$ exists, is $\leq A'^2$, and can be computed in time $\Ostar(1)$ as before.
By Heuristic~\ref{heur:deathattack-kate} Part~\ref{heurdeath-kate-pt1}, 
for every choice of $d'$ as we iterate there exists a $c < p$ in Step~\ref{step:find-c}
with probability 1/2,
hence the smallest $c$ found, the one output by Algorithm~\ref{alg:solve_eqn}, 
will have size $\approx A^{-\delta}p$.
Also, we have $d' \leq A^{\delta}$ so $d^2 \leq A'^4 A^{2\delta}$.
Now \eqref{death-ineq} ensures that
$eB^2 > d^2 + c^2A'^2$,
so the quantity in Heuristic~\ref{heur:deathattack-kate} Part~\ref{heurdeath-kate-pt2} 
(with $A'$ in place of $A$) 
is positive;
hence, since $B$ is at most polynomial in $A$, 
under Heuristic~\ref{heur:deathattack-kate}, Step~\ref{step:cornacchia} succeeds after 
$\Ostar(1)$ iterations.
Hence Algorithm~\ref{alg:solve_eqn} terminates in time $\Ostar(A^{\delta})$ classically
or time $\Ostar(A^{\frac{\delta}{2}})$ quantumly using Grover's algorithm.
\qed
\end{proof}

\begin{proof}[of Proposition~\ref{prop:nonpolydual}]
\newcommand\qfact{\ensuremath{f}}
    Write $\qfact=1$ for classical algorithms
    and $\qfact=\frac12$ for quantum algorithms;
    hence,
    the complexity of Algorithm~\ref{alg:hybrid} as discussed in Subsection~\ref{subsec:nonpoly}
    equals
    $\comp=\max\{\qfact\delta,\qfact\gamma+\frac12\epsilon\}$.
    Call a tuple $(\delta,\gamma,\epsilon)\in[0;1]^3$
    \enquote{admissible} if it
    satisfies the bounds
    \begin{align*}
        \tag{$\ast$}
        \label{eq:admissible_bounds}
        (4+2\delta-4\gamma - \epsilon)\alpha \leq 2\beta
        \quad \text{and} \quad
        (2-2\delta-2\gamma - \epsilon)\alpha \leq 2\beta - 2
    \end{align*}
    from Lemma~\ref{lem:eqdeathsol}.
    Suppose given an admissible tuple $(\delta,\gamma,\epsilon)$
    with cost $\comp\leq1/2$.
    First, notice
    that setting $\gamma':=\max\{\delta,\gamma+\frac1{2f}\epsilon\}$,
    the tuple $(\delta,\gamma',0)$
    is still admissible
    with the same $\comp$.
    (Since $\comp\leq1/2$, we have $\gamma'\leq\frac1{2f}\leq1$.)
    Thus, it suffices to
    consider
    admissible
    tuples $(\delta,\gamma',0)$
    with $0\leq\delta\leq\gamma'\leq1$
    when optimizing.
    The bounds~\eqref{eq:admissible_bounds}
    simplify to
    \begin{align*}
        \tag{$\ast'$}
        \label{eq:admissible_bounds2}
        1+\alpha-\beta-\alpha\gamma'
        \,\,\leq\,\,
        \alpha\delta
        \,\,\leq\,\,
        \beta-2\alpha+2\alpha\gamma'
        \,\text,
    \end{align*}
    which
    (leaving out the middle term $\alpha\delta$ and simplifying)
    implies
    \begin{equation}
        \label{eq:bound1}
        \tag{$\ast_1$}
        \gamma' \,\,\geq\,\, \frac{1+3\alpha-2\beta}{3\alpha}
        \text.
    \end{equation}
    This establishes a lower bound on $\gamma'$,
    but it is not yet clear which of these values
    are actually possible:
    For a given $\gamma'$,
    we additionally require a
    $\delta\in[0;\gamma']$ that satisfies
    the bounds~\eqref{eq:admissible_bounds2}.
    Hence,
    the upper bound
    $\beta-2\alpha+2\alpha\gamma'$
    on $\alpha\delta$
    in~\eqref{eq:admissible_bounds2}
    must be non-negative,
    which simplifies to
    \begin{equation}
        \label{eq:bound2}
        \tag{$\ast_2$}
        \gamma' \,\,\geq\,\, \frac{2\alpha-\beta}{2\alpha}
        \,\text.
    \end{equation}
    Similarly,
    the lower bound
    $1+\alpha-\beta-\alpha\gamma'$
    on
    $\alpha\delta$
    in~\eqref{eq:admissible_bounds2}
    must not be greater than~$\alpha\gamma'$, yielding
    \begin{equation}
        \label{eq:bound3}
        \tag{$\ast_3$}
        \gamma' \,\,\geq\,\, \frac{1+\alpha-\beta}{2\alpha}
        \,\text.
    \end{equation}
    Recalling that $\comp=f\gamma'$,
    this shows the claim.
\qed
\end{proof}

\subsection{Proof of Proposition~\ref{prop:nonpolyfrob}}\label{proof-nonpolyfrob}

\begin{proof}[of Proposition~\ref{prop:nonpolyfrob}]
In the precomputation step of Algorithm \ref{alg:frob-eqn}, 
we use $A' = A ^{1-\gamma} = p^{\alpha - \alpha \gamma}$ as an input in place of $A$.
By Proposition~\ref{prop:eq2}, under Heuristic~\ref{heur:deathattack-kate2},
if $B > \sqrt{p}A'^2$ but is at most polynomial in $A$ and $A'$ has $O(\log\log p)$ distinct prime factors,
then
this step returns $(a,b,c,e)$ with $e = O(\log p)$ in polynomial time.
Note that if $\gamma > (1+4\alpha-2\beta)/4\alpha$ then
\[B = p^{\beta} > p^{\frac{1+4\alpha-4\alpha\gamma}{2}} = \sqrt{p}A'^2,\]
that is, for $\gamma > (1+4\alpha-2\beta)/4\alpha$, the precomputation step runs in polynomial time.
The online step runs in time $\Ostar(A^{\gamma})$ on a classical computer and 
time $\Ostar(A^{\frac{\gamma}{2}})$ on a quantum computer using Grover's algorithm,
so the result now follows by choosing $\gamma$ optimally for any given $\alpha$ and $\beta$. \CP{this proves the proposition, but not optimality. Is $\epsilon=0$ optimal?}
\qed
\end{proof}

\subsection{Proof of Proposition~\ref{extend-improve}}\label{proof-nonpolytrapdoor}

\begin{proof}[of Proposition~\ref{extend-improve}]
With the modifications described in Subsection~\ref{subsec:nonpolytrap},
the attacker invokes
Algorithm~\ref{alg:find_isog} to compute the secret isogeny, 
using the data~$(\theta,d,e)$ from Algorithm~\ref{alg:insecure}.

We analyze the complexity of running the modified Algorithm~\ref{alg:insecure} followed by~Algorithm \ref{alg:find_isog}.
The two quadratic residuosity conditions are heuristically satisfied one in four times,
so we ignore them in this analysis.
The cost of Algorithm~\ref{alg:insecure} modified in this way becomes 
$\Ostar(A^{\epsilon'})$ for a classical adversary and $\Ostar(A^{\frac{\epsilon'}{2}})$
for a quantum adversary.

Note also that by construction we have $e \leq A^{\epsilon}$, 
so the cost of running Algorithm~\ref{alg:find_isog} 
will be $\Ostar(A^{\gamma + \frac{\epsilon}{2}})$ for a classical adversary and
$\Ostar(A^{\frac{\gamma + \epsilon}{2}})$ for a quantum adversary,
following the same reasoning as in the complexity analysis of Algorithm~\ref{alg:hybrid}.

We now look at the conditions for existence of a solution in Algorithm~\ref{alg:insecure}.
Note that $d$ is a priori bounded by $(A')^2=A^{2(1-\gamma)}$. 
However, after trying $A^{\epsilon}$ values for $e$ we may hope to find some $d$ bounded by $A^{2(1-\gamma)-\epsilon}$. 
To satisfy \eqref{sign-cond} we need 
$$2\beta>\alpha(4-4\gamma-2\epsilon'-\epsilon),$$
and by construction
we also need $\epsilon'\leq\epsilon$.

For a classical adversary, setting $\epsilon=\epsilon' = 2\gamma = \comp$ gives the result.
For a quantum adversary, setting $\epsilon=\epsilon' = 0$ and $\gamma = 2\cdot \comp$ gives the result.
\qed
\end{proof}

\begin{remark}
We found these choices for $\epsilon,\epsilon',\gamma$ by solving the following optimization problems for $\alpha = \beta = \frac{1}{2}$, so at least in that case (which corresponds to SIKE) we expect there to be no better choice with respect to overall complexity:
For the best classical attack when $\alpha=\beta=\frac{1}{2}$ we solved the following linear optimization problem:%
$$\min_{\substack{
4\gamma+2\epsilon'+\epsilon \geq 2, \\
\epsilon \geq \epsilon'
}} 
\max\left\{\epsilon',\gamma+\frac{\epsilon}{2}\right\}.$$

\noindent
For the best quantum attack when $\alpha=\beta=\frac{1}{2}$ we solved the following linear optimization problem:
$$\min_{\substack{
4\gamma+2\epsilon'+\epsilon \geq 2 \\
\epsilon \geq \epsilon'
}}
\max\left\{\frac{\epsilon'}{2},\frac{\gamma+\epsilon}{2}\right\}.$$
\end{remark}

%% file: more-examples.tex
\section{Additional examples of backdoored primes} \label{Appendix}

In the examples in Subsection~\ref{ss:Insecure p}, we let $\None=2^{216}$, $\Ntwo=3^{300}, e=1$. We let $d$ equal $\Ntwo\bmod\None^2$, and $D=\frac{\Ntwo-d^2}{\None^2}$, hence
\begin{align*}
D = \frst 16896420333246701930066245846797285820453043046692612
    \cont 34160275705261296847619733634147787139416180071370253
    \cont 151875694583397987452872630971686172791991823800180
    \text.
\end{align*}
We first choose $c=53$, then $D-c^2$ is a prime number (i.e., $a=1,~b=0$),
\begin{align*}
p = \frst 16896420333246701930066245846797285820453043046692612
    \cont 34160275705261296847619733634147787139416180071370253
    \cont 151875694583397987452872630971686172791991823797371
    \text.
\end{align*}
When $c=355$, then $D-c^2$ is 5 times a prime number, namely,
\begin{align*}
p = \frst 33792840666493403860132491693594571640906086093385224
    \cont 68320551410522593695239467268295574278832360142740506
    \cont 30375138916679597490574526194337234558398364734831
    \text.
\end{align*}
Both of these primes are congruent to $3$ modulo $4$.

\medskip

We also give additional examples of Pythagorean triples as described in Subsection \ref{subsec:param}.
In particular, let
\begin{align*}
\Ntwo &= 17^{60}, \\
\None &= 2^5 \cdot 3^2 \cdot 5^2 \cdot 7 \cdot 11 \cdot 13 \cdot 19 \cdot 23 \cdot 41 \cdot 47 \cdot 59 \cdot 61 \cdot 101 \cdot 181 \cdot 191 \cdot 199 \cdot 239 \cdot 421 \\
&\hspace{2em}  \cdot 541 \cdot 659 \cdot 769 \cdot 2281 \cdot 16319 \cdot 30119 \cdot 285599 \cdot 391679 \cdot 1039081 \cdot 1109159 %\\
%d &= 79 \cdot 241 \cdot 401 \cdot 5279 \cdot 71039 \cdot 238081 \cdot 3109921 \cdot 279323519 \cdot 4004030399 \\
%&\hspace{2em} \cdot 29822041919 \cdot 411011888260834799 
\end{align*}
For this, $177\None\Ntwo-1 \equiv 3 \pmod 4$ is prime.
\noindent
Finally, a powersmooth example is given by
\begin{align*}
\Ntwo &= 5^8 \cdot 13^4 \cdot 17^4 \cdot 29^4 \cdot 37^4 \cdot 41^4 \cdot 53^4 \cdot 61^4 \cdot 73^4 \cdot 89^4 \cdot 97^4, \\
\None &= 2^4 \cdot 3 \cdot 7 \cdot 11 \cdot 23 \cdot 31 \cdot 127 \cdot 199 \cdot 811 \cdot 2903 \cdot 155383 \cdot 842041\cdot 933199 \cdot 1900147 \\
&\hspace{2em} \cdot 8333489 \cdot 21629743 \cdot 30583723 \cdot 69375497%, \\
%d &= 113 \cdot 353 \cdot 3583 \cdot 70913\cdot 1719217 \cdot 1180340726607311 \cdot 1394472264051281903 \\
%&\hspace{2em} \cdot 2040375971230991663
\end{align*}
For this, $19\None\Ntwo-1 \equiv 3 \pmod 4$ is prime.

%% file: impl.tex
\section{Computing concrete backdoor instances}\label{imp}

In this section we report on computations regarding Algorithm \ref{alg:insecure} for some concrete parameters. We chose parameters $\None=2^{216},~\Ntwo=3^{300},~p=\None\Ntwo\cdot 277-1$. It is easy to see that we can choose $e=1$ and $d$ equal to $\Ntwo$ modulo $\None^2$. Now we need to factor $\frac{\Ntwo^2-d^2}{\None^2}$. The way we chose $d$ makes it easy as $\frac{\Ntwo^2-d^2}{\None^2}=\frac{\Ntwo-d}{\None^2}(\Ntwo+d)$. This is something which applies in other cases as well, and to make sure that factorization is easy one can try choices of $d$ until factoring $\Ntwo+d$ is feasible (e.g., $\Ntwo+d$ is a prime number). For completeness, the factorization of $\frac{\Ntwo^2-d^2}{\None^2}$ is
\begin{align*}
    &\,\,
    2^2\cdot 5\cdot 23 \cdot 359\cdot 2089 \cdot 39733 \cdot 44059 \cdot 74353 \cdot \mbox{}
    \\
    \frst 37628724343042581190433455539389264355404578964704347
    \cont 59039416676945740598806299461624575502089058332472952
    \cont 9427908921244148421914499463
    \text.
\end{align*}

\vspace{1 mm}

Once the factorization is known, we apply Simon's algorithm, implemented in Pari/GP~\cite{batut2000user}
as \texttt{qfsolve()}, to compute a rational solution to the equation $pa^2+pb^2+c^2=\frac{\Ntwo^2-d^2}{\None^2}$. A rational solution is given by
\begin{align*}
    a =
    \frst 32319123496536786843254458765608553095663568521872334
    \cont 297530315749275438736572
            / z
    \\
    b =
    \frst 37902893736016880777193854875253045553175457573067191
    \cont 2406340378400674751175560
            / z
    \\
    c =
    \frst 85437128777417136022423941321585505761757160615798739
    \cont 72406075696054195168847143870020389324092617191284723
    \cont 80905798835064955553407208320599901478282089806543945
    \cont 266931422175906643935346
            / z
    \text,
    \\
    \text{where}
    \\
    z =
    \frst 87978348577011335417453239649099382225650021375809220
    \cont 4820354441211407993264179570949123846469170675585119
    \text.
\end{align*}

Once $\theta$ is computed one has to compute an order $\OO_0$ which contains $\theta$. This can be accomplished in various ways. 
One way is to find a $\theta'$ such that $\theta\theta'+\theta'\theta=0$ and $\theta'^2$ is an integer multiple of the identity. 
This amounts to finding the kernel of the linear map $\eta\mapsto \theta\eta+\eta\theta$, 
which is a 2-dimensional vector space over $\mathbb{Q}$ (i.e., one chooses an element in this kernel and then multiplies it with a suitable integer).
It is preferable to construct $\OO_0$ in this way so that 
the discriminant of the order is the square of the reduced norm of $\theta\theta'$.
In particular, if we choose a $\theta'$ whose norm is easy to factor, 
then the discriminant is also easy to factor. 
One has a lot of flexibility in choosing $\theta'$ and lattice reduction techniques help finding one which is sufficiently small and has an easy factorization. 
Note that the norm of $\theta'$ will always be divisible by $p$ since the discriminant of every order is a multiple of $p$ (and the norm of $\theta$ is coprime to $p$). 
Finally, one can compute a maximal order containing $\OO_0$
using MAGMA's~\cite{bosma1997magma} \texttt{MaximalOrder()} function.